\documentclass[11pt,a4paper]{article}

\usepackage[T1]{fontenc}
\usepackage[utf8]{inputenc}
\usepackage{lmodern}
\usepackage{microtype}
\usepackage{amsmath,amssymb,amsfonts,amsthm,mathtools}
\usepackage{bm}
\usepackage{booktabs,subcaption}
\usepackage{enumitem}
\usepackage{hyperref}
\usepackage{cite}
\usepackage[capitalise,nameinlink]{cleveref}
\usepackage{xcolor}

\usepackage[a4paper,margin=3cm]{geometry}
\setlist[itemize]{leftmargin=2em}

\usepackage[mathlines]{lineno}
\usepackage{placeins}

\newtheorem{theorem}{Theorem}[section]
\newtheorem{proposition}[theorem]{Proposition}
\newtheorem{lemma}[theorem]{Lemma}
\newtheorem{corollary}[theorem]{Corollary}
\theoremstyle{definition}

\newtheoremstyle{remarkbold}
{}% Space above
{}% Space below
{\itshape}% Emphasized text (\itshape)
{}% Indentation
{\bfseries}% Style of title (\bfseries)
{.}% Full stop in the title
{ }% Space in the title
{}% Specifica del titolo (opzionale)

\theoremstyle{remarkbold}
\newtheorem{remark}[theorem]{Remark}

\title{Ensemble Kalman Inversion as an Inertial Interacting Particle System}
\author{
	Michael Herty \\
	\small \it{Institute for Geometry and Applied Mathematics, RWTH Aachen University} \\ \small \it{and} \\ \small \it{Department of Mathematics and Applied Mathematics, University of Pretoria}
	\and
	Pierpaolo Porretta \\
	\small \it{Department of Mathematics, Sapienza University of Rome}
	\and
	Giuseppe Visconti \\
	\small \it{Department of Mathematics, Sapienza University of Rome}
}
\date{\today}

\begin{document}	
	
\maketitle

\begin{abstract}
	{\color{black}
		Ensemble Kalman Inversion (EKI) is a derivative-free, ensemble-based method	for inverse and optimization problems. A central limitation of its continuous-time dynamics is premature covariance collapse, which can restrict the directions accessible to the inversion and increase sensitivity to the	initial ensemble.
		
		We introduce a second-order interacting particle system that combines the Kalman force with inertia, damping, attraction towards the ensemble mean, and	short-range repulsion. We establish a generalized subspace property showing that nonzero initial velocities may enlarge the affine space accessible to the dynamics, while a finite-ensemble restriction remains. For linear inverse
		problems, we derive the mean and fluctuation dynamics, identify a regime in	which fully collapsed configurations are linearly unstable, characterize asymptotic optimality on the subspace retained by the limiting covariance, and prove exponential decay for the associated frozen-covariance mean dynamics.
		
		Numerical experiments on high-dimensional linear and nonlinear inverse problems investigate subspace enlargement, discrepancy-based regularization, and ensemble collapse.
		}
\end{abstract}



\section{Introduction}
\label{sec:introduction}

Inverse problems arise in many areas of applied mathematics, science and
engineering, where unknown parameters or states must be recovered from
indirect and typically noisy observations. Their stable solution generally
requires regularization; see, for instance,
\cite{EnglHankeNeubauer1996,Hansen1998,Stuart2010,dashtistuart2017}.
In the Bayesian framework, the solution is instead described by a posterior
distribution combining data and prior information
\cite{Stuart2010,dashtistuart2017}.

Ensemble Kalman methods provide a derivative-free framework for data
assimilation and inverse problems. Starting from the Ensemble Kalman Filter
\cite{Evensen1994}, ensemble-based Kalman methods have been applied in
geophysics, reservoir engineering, and uncertainty quantification
\cite{Aanonsen2009,LawStuart2012,iglesiaslawstuart2013b}. Ensemble Kalman
Inversion (EKI) was introduced for inverse problems in
\cite{iglesiaslawstuart2013}. Its key feature is the use of empirical
covariances to define the update, avoiding explicit derivatives of the forward
map and making the method attractive for large-scale or black-box problems.
Convergence and regularization properties have been studied in linear and
noisy settings
\cite{schillingsstuart2017,schillingsstuart2018,SchillingsPreprint}, together
with Tikhonov regularization \cite{ChadaStuartTong2019}, constrained and
regularized variants
\cite{ChadaShillingsWeissmann2019,Stuart2019CEnKF,HertyVisconti2020,
	ZhangStroferXiao2019}, hierarchical approaches \cite{ChadaIglesiasRoininenStuart2018}, machine
learning applications \cite{KovachkiStuart2019,Haber2018NeverLB}, and
mean-field and continuous-time formulations
\cite{lawtembinetempone2016,DingLi2019,DingLiLu,BungertWacker2023,
	HertyVisconti2019,CarrilloVaes}.

From a dynamical systems viewpoint, continuous-time EKI can be interpreted as
an interacting particle system driven by a covariance-preconditioned descent
direction. In the linear setting, the empirical covariance adapts the
dynamics to the ensemble geometry, but it may also collapse prematurely.
Rank loss then restricts the directions available to the inversion and may
lead to stationary configurations that do not correspond to the desired
solution. Stabilization mechanisms based on covariance inflation and
relaxation have therefore been proposed
\cite{ArmbrusterHertyVisconti2022}.

In parallel, acceleration mechanisms for EKI have been inspired by classical
inertial optimization methods such as Polyak's heavy-ball method
\cite{Polyak1964} and Nesterov acceleration
\cite{Nesterov1983,SuBoydCandes2016}. Nesterov-type acceleration has recently
been incorporated into EKI through a particle-level nudging step
\cite{VernonBachDunbar2025}. The particle interpretation of EKI also connects
with interacting particle methods for optimization, including Particle Swarm
Optimization and Consensus-Based Optimization
\cite{EberhartKennedy1995,ShiEberhart1998,WangTanLiu2017,
	PinnauTotzecTseMartin2016,CarrilloJinLIZhu2020}, as well as related
mean-field, memory, polarization, and micro--macro developments
\cite{GrassiPareschi2021,BorghiGrassiPareschi2023,
	BungertRoithWacker2025,BungertHoffmannKimRoith2025,HertyVeneruso2025}.

{\color{black}
In this paper we introduce an inertial Ensemble Kalman particle system for
inverse and optimization problems. Our aim is not simply to accelerate the
minimization of the data-misfit functional, but to modify the internal
ensemble dynamics so as to counteract premature contraction and to alter the
geometry of the directions accessible to the inversion. The method extends
continuous-time EKI by assigning positions and velocities to the particles
and by coupling the Kalman-type force with damping, attraction towards the
ensemble mean, and short-range repulsion. Unlike Nesterov-type accelerations,
which modify the discrete ensemble update, the proposed model is formulated
directly as a continuous-time second-order interacting particle system.

A first structural result is a generalized subspace property: the
second-order trajectories remain in the affine space generated jointly by the
initial position anomalies and initial velocities. Thus nonzero initial
velocities may enlarge the accessible space compared with standard EKI,
although the finite-ensemble subspace restriction is not removed. For linear
inverse problems, we derive the mean and fluctuation dynamics and show that
fully collapsed configurations are linearly unstable with respect to
zero-mean fluctuation perturbations in a suitable parameter regime. Assuming
convergence of the full dynamics to an asymptotic equilibrium, we characterize
the limiting mean through an optimality condition on the subspace retained by
the limiting covariance. We then study the associated frozen-covariance mean
equation and prove exponential decay along the retained directions. These
results describe structural mechanisms of the second-order dynamics and do
not imply prevention of covariance collapse for the full dynamics.

The numerical experiments investigate these mechanisms in inverse-problem
regimes with noisy observations and finite ensembles. We verify
accessible-subspace enlargement in a controlled linear problem and under
independent high-dimensional prior sampling, study discrepancy-stopped
reconstructions as the noise level decreases, and consider a nonlinear
high-dimensional Darcy inverse problem. The latter is used to examine
subspace enlargement, the separate roles of attraction and repulsion, and
dependence on the ensemble size, including comparisons with standard and
span-preserving inflated EKI and an assessment of computational work.
}

\section{From continuous-time EKI to an inertial particle system}
\label{sec:continuous-eki-inertial-model}

\subsection{Continuous-time EKI}
\label{subsec:continuous-time-eki}

We recall the continuous-time formulation of EKI and the main structural features that motivate the second-order dynamics introduced below. Let
\[
y = G(u^\dagger) + \eta
\]
be the observed data, where \(u^\dagger \in \mathbb{R}^d\) is the unknown parameter, \(G:\mathbb{R}^d \to \mathbb{R}^K\) is the continuous forward map, and \(\eta\) denotes observational noise. We assume that \(\eta\sim\mathcal{N}(0,\Gamma)\), where \(\Gamma \in \mathbb{R}^{K\times K}\) is a symmetric positive definite covariance matrix, and we consider the least-squares functional
\begin{equation}
	\Phi(u)
	=
	\frac12 \|y-G(u)\|_{\Gamma}^{2}.
	\label{eq:least-squares-functional}
\end{equation}
Here and in the following, \(\|z\|_\Gamma^2 := z^T\Gamma^{-1}z\).

Let \(U(t)=\{u_j(t)\}_{j=1}^J\) be an ensemble of \(J\) particles in \(\mathbb{R}^d\). We denote by
\[
\bar u(t) = \frac1J\sum_{j=1}^J u_j(t),
\qquad
\bar G(t) = \frac1J\sum_{j=1}^J G(u_j(t))
\]
the empirical means in parameter and observation space. The empirical cross-covariance between parameters and model outputs is defined by
\begin{equation}
	C^{uG}(U(t))
	=
	\frac1J\sum_{j=1}^J
	(u_j(t)-\bar u(t))\otimes (G(u_j(t))-\bar G(t))
	\in \mathbb{R}^{d\times K}.
	\label{eq:cross-covariance}
\end{equation}
The continuous-time EKI dynamics \cite{schillingsstuart2017} is then given by
\begin{equation}
	\dot u_j(t)
	=
	C^{uG}(U(t))\Gamma^{-1}(y-G(u_j(t))),
	\qquad j=1,\dots,J.
	\label{eq:continuous-eki}
\end{equation}
This system can be interpreted as an interacting particle method in which the drift of each particle is determined by empirical covariance information extracted from the ensemble.

In the linear case \(G(u)=Gu\), with \(G\in\mathbb R^{K\times d}\), the least-squares functional satisfies
\begin{equation}
\nabla\Phi(u)=G^T\Gamma^{-1}(Gu-y)=Au-b,
\qquad
A:=G^T\Gamma^{-1}G,\quad b:=G^T\Gamma^{-1}y .
\label{eq:Phi-A-b-definitions}
\end{equation}
Moreover,
\[
C^{uG}(U(t))=C(U(t))G^T,
\]
where
\[
C(U(t))
=
\frac1J\sum_{j=1}^J
(u_j(t)-\bar u(t))\otimes(u_j(t)-\bar u(t))
\]
is the empirical covariance. Hence the continuous-time EKI dynamics reduces to
\begin{equation}
	\dot u_j(t)
	=
	-C(U(t))(Au_j(t)-b),
	\qquad j=1,\dots,J.
	\label{eq:linear-continuous-eki}
\end{equation}
Thus, in the linear setting, EKI is a covariance-preconditioned gradient flow. The dynamics is still nonlinear, since the preconditioner \(C(U(t))\) depends on the evolving ensemble.

In the linear case, averaging \eqref{eq:linear-continuous-eki} gives the mean dynamics
\begin{equation}
	\dot{\bar u}(t)
	=
	-C(U(t))(A\bar u(t)-b).
	\label{eq:first-order-mean-dynamics}
\end{equation}
If \(e_j(t)=u_j(t)-\bar u(t)\), then the fluctuations satisfy
\begin{equation}
	\dot e_j(t)
	=
	-C(U(t))A e_j(t),
	\qquad j=1,\dots,J.
	\label{eq:first-order-fluctuation-dynamics}
\end{equation}
These identities show that the evolution of the mean is entirely mediated by the empirical covariance, while the fluctuations determine the directions retained by the ensemble.

A fundamental structural property of EKI is the subspace property: the dynamics remains confined to the affine space generated by the initial ensemble,
\begin{equation}
	u_j(t)\in \bar u(0)+\operatorname{span}\{e_k(0):k=1,\dots,J\},
	\qquad j=1,\dots,J.
	\label{eq:subspace-property}
\end{equation}
Moreover, if \(C(U_*)=0\), then all particles coincide and the right-hand side of
\eqref{eq:linear-continuous-eki} vanishes. Thus every fully collapsed ensemble is a stationary configuration, independently of whether its common value minimizes \(\Phi\). More generally, if the covariance loses rank, the dynamics can only move along the directions retained by the ensemble covariance. At a stationary configuration \(U_*\), \eqref{eq:first-order-mean-dynamics} gives
\[
C(U_*)(A\bar u_*-b)=0,
\]
so that the gradient of \(\Phi\) is required to vanish only on the range of the empirical covariance. This covariance-collapse mechanism motivates the inertial interacting particle system introduced below.

\subsection{The inertial particle system}
\label{subsec:inertial-particle-system}

In the previous section the EKI particles were denoted by \(u_j(t)\), consistently with the inverse problem variable \(u\). In the second-order formulation it is convenient to use a mechanical notation. We therefore denote by
\[
X(t)=\{x_j(t)\}_{j=1}^J,
\qquad
V(t)=\{v_j(t)\}_{j=1}^J
\]
the particle positions and velocities, respectively. The position \(x_j(t)\in\mathbb{R}^d\) plays the role of the parameter particle \(u_j(t)\) in first-order EKI.

We introduce a second-order particle system of the form
\begin{equation}
	\begin{cases}
		\dot x_j(t)=v_j(t),\\[0.4em]
		\dot v_j(t)
		=
		a_j^{\rm EKI}(X(t),V(t))
		+
		a_j^{\rm rep}(X(t))
		+
		a_j^{\rm att}(X(t)),
	\end{cases}
	\qquad j=1,\dots,J.
	\label{eq:second-order-abstract}
\end{equation}
The three terms in the acceleration represent, respectively, a Kalman-type relaxation force, a short-range repulsive interaction, and an attraction force towards the ensemble mean.

As before, we define
\[
\bar x(t)=\frac1J\sum_{j=1}^J x_j(t),
\qquad
\bar G(t)=\frac1J\sum_{j=1}^J G(x_j(t)),
\]
and the empirical cross-covariance
\begin{equation*}
	C^{xG}(X(t))
	=
	\frac1J\sum_{j=1}^J
	(x_j(t)-\bar x(t))\otimes (G(x_j(t))-\bar G(t)).
	%\label{eq:second-order-cross-covariance}
\end{equation*}
The Kalman-type relaxation force is defined by
\begin{equation*}
	a_j^{\rm EKI}(X(t),V(t))
	=
	-\gamma v_j(t)
	+
	\beta C^{xG}(X(t))\Gamma^{-1}(y-G(x_j(t))),
	%\label{eq:second-order-eki-force}
\end{equation*}
where \(\gamma>0\) is a damping coefficient and \(\beta>0\) controls the strength of the EKI correction. The damping term dissipates kinetic energy, while the second term is the natural Kalman-type force inherited from continuous-time EKI.

The repulsive interaction is introduced to counteract premature collapse of the ensemble. We define
\begin{equation*}
	a_j^{\rm rep}(X(t))
	=
	k\sum_{i\neq j}
	f(\|x_i(t)-x_j(t)\|)
	(x_j(t)-x_i(t)),
	%\label{eq:second-order-repulsion}
\end{equation*}
where \(k\geq0\) is the repulsion strength and \(f:[0,\infty)\to(0,\infty)\) is a smooth non-increasing interaction kernel. In the numerical experiments and in the stability discussion below, we use the regularized inverse-power kernel
\begin{equation}
	f(r)=\frac{1}{(\varepsilon+r)^p},
	\qquad
	\varepsilon>0,\quad p>1.
	\label{eq:repulsive-kernel}
\end{equation}
This choice produces a strong short-range repulsion while remaining bounded at \(r=0\). The parameter \(\varepsilon\) controls the regularization near the origin, whereas \(p\) controls the decay of the interaction at large distances.

{\color{black}
	\begin{remark} \label{rem:normalization:repulsion}
		When comparing ensembles with different cardinalities, it is natural to
		account for the number of terms in the pairwise interaction sum. We will
		therefore also consider the ensemble-size normalized parametrization
		\[
		k=k_J:=\frac{\kappa}{J},
		\]
		where \(\kappa\geq0\) is independent of \(J\). We retain \(k\) in the
		formulation and analysis below in order to allow for a general repulsion
		strength.	
	\end{remark}
}

The attraction force towards the ensemble mean is defined by
\begin{equation*}
	a_j^{\rm att}(X(t))
	=
	-\alpha(x_j(t)-\bar x(t)),
	%\label{eq:second-order-attraction}
\end{equation*}
with \(\alpha\geq0\). Its role is to prevent the ensemble from dispersing excessively. The proposed dynamics is therefore based on a competition between attraction and repulsion: {\color{black} the attraction term promotes collective coherence, while the repulsive
term counteracts contraction of the ensemble at short distances.}

Combining the above definitions we obtain the second-order EKI system
\begin{equation}
	\begin{cases}
		\dot x_j(t)=v_j(t),\\[0.4em]
		\begin{aligned}
			\dot v_j(t)
			&=
			-\gamma v_j(t)
			+
			\beta C^{xG}(X(t))\Gamma^{-1}(y-G(x_j(t))) \\
			&\quad
			+
			k\displaystyle\sum_{i\neq j}
			f(\|x_i(t)-x_j(t)\|)
			(x_j(t)-x_i(t))
			-
			\alpha(x_j(t)-\bar x(t)),
		\end{aligned}
	\end{cases}
	\label{eq:second-order-eki}
\end{equation}
for \(j=1,\dots,J\).

The model can be interpreted as an inertial, heavy-ball-type reformulation of continuous-time EKI. In the overdamped regime \(\gamma\gg1\), with \(\beta/\gamma=\mathcal{O}(1)\), the velocity equation formally relaxes to a first-order Kalman-type drift. At the same time, the additional attraction--repulsion mechanism changes the internal geometry of the ensemble and is designed to reduce the tendency of first-order EKI to collapse prematurely.

{\color{black}
	
	\begin{proposition}[Generalized subspace property]
		\label{prop:generalized-subspace}
		Let $(X(t),V(t))$ be a classical solution of \eqref{eq:second-order-eki}.
		Define
		\[
		\tilde x_j(t):=x_j(t)-\bar x(t),
		\qquad
		\tilde v_j(t):=v_j(t)-\bar v(t),
		\]
		and let
		\[
		\mathcal S_0
		:=
		\operatorname{span}
		\left\{
		\tilde x_j(0),\tilde v_j(0):
		j=1,\ldots,J
		\right\}.
		\]
		Then
		\[
		\tilde x_j(t),\tilde v_j(t)\in\mathcal S_0,
		\qquad j=1,\ldots,J,
		\]
		for every time for which the solution exists.
		
		Moreover, defining
		\[
		\mathcal V_0
		:=
		\mathcal S_0+\operatorname{span}\{\bar v(0)\}
		=
		\operatorname{span}
		\left\{
		\tilde x_j(0),v_j(0):
		j=1,\ldots,J
		\right\},
		\]
		we have
		\[
		x_j(t)\in \bar x(0)+\mathcal V_0,
		\qquad
		v_j(t)\in\mathcal V_0,
		\qquad j=1,\ldots,J.
		\]
	\end{proposition}
	
	\begin{proof}
		First observe that, for every $q\in\mathbb R^K$,
		\[
		C^{xG}(X)q
		=
		\frac1J\sum_{\ell=1}^J
		\tilde x_\ell
		\bigl(G(x_\ell)-\bar G\bigr)^T q,
		\]
		and therefore
		\[
		C^{xG}(X)q
		\in
		\operatorname{span}
		\{\tilde x_1,\ldots,\tilde x_J\}.
		\]
		
		Averaging the velocity equations in \eqref{eq:second-order-eki}, the attraction
		and repulsion terms cancel, and we obtain
		\[
		\dot{\bar v}
		=
		-\gamma\bar v
		+
		\beta C^{xG}(X)\Gamma^{-1}(y-\bar G).
		\]
		Subtracting this equation from the equation for $v_j$ gives
		\[
		\begin{aligned}
			\dot{\tilde x}_j
			&=
			\tilde v_j,\\
			\dot{\tilde v}_j
			&=
			-\gamma\tilde v_j
			+
			\beta C^{xG}(X)\Gamma^{-1}
			\bigl(\bar G-G(x_j)\bigr)
			-\alpha\tilde x_j\\
			&\quad
			+
			k\sum_{i\neq j}
			f(\|\tilde x_i-\tilde x_j\|)
			(\tilde x_j-\tilde x_i).
		\end{aligned}
		\]
		Every term on the right-hand side belongs to the linear span of the current
		position and velocity fluctuations. Since initially
		$\tilde x_j(0),\tilde v_j(0)\in\mathcal S_0$, the subspace
		$\mathcal S_0$ is invariant under the fluctuation dynamics. Hence
		\[
		\tilde x_j(t),\tilde v_j(t)\in\mathcal S_0.
		\]
		
		The mean equation and the inclusion
		\[
		\operatorname{Range}(C^{xG}(X(t)))
		\subseteq \mathcal S_0
		\]
		then imply
		\[
		\bar v(t)\in
		\mathcal S_0+\operatorname{span}\{\bar v(0)\}
		=\mathcal V_0.
		\]
		Since $\dot{\bar x}=\bar v$,
		\[
		\bar x(t)-\bar x(0)
		=
		\int_0^t\bar v(s)\,\mathrm{d}s
		\in\mathcal V_0.
		\]
		Finally,
		\[
		x_j(t)-\bar x(0)
		=
		\bar x(t)-\bar x(0)+\tilde x_j(t)
		\in\mathcal V_0,
		\]
		and
		\[
		v_j(t)=\bar v(t)+\tilde v_j(t)\in\mathcal V_0.
		\]
		This proves the result.
	\end{proof}
	
	\begin{remark}[Role of the initial velocities]
		\label{rem:initial-velocities-subspace}
		If $v_j(0)=0$ for all $j$, then
		\[
		\mathcal V_0
		=
		\operatorname{span}
		\{\tilde x_j(0):j=1,\ldots,J\},
		\]
		and the usual EKI affine-subspace restriction is recovered. In this case,
		the attraction--repulsion mechanism can redistribute the particles and
		increase their spread, but it cannot generate directions that are absent
		from the initial ensemble.
		
		More generally, if some initial velocities have components outside
		$\operatorname{span}\{\tilde x_j(0)\}$, then the accessible subspace
		may be strictly enlarged. Indeed,
		\[
		\dim(\mathcal V_0)
		\leq
		\min\{d,2J-1\},
		\]
		whereas the subspace generated by the initial position anomalies has
		dimension at most $J-1$. Thus nonzero initial velocities may enlarge the
		set of directions explored by the inertial dynamics, although a
		finite-dimensional subspace restriction remains.
	\end{remark}
}

The analysis below\textcolor{black}{, instead,} focuses on the linear case for which the Kalman force reduces to a covariance-preconditioned gradient term.

\section{Analysis in the linear inverse problem setting}
\label{sec:analysis-linear-model}

We analyze the inertial EKI system introduced in Section \ref{subsec:inertial-particle-system} in the linear inverse problem setting. Thus \(G(x)=Gx\), and
\[
\nabla\Phi(x)=Ax-b,
\qquad
A=G^T\Gamma^{-1}G,\quad b=G^T\Gamma^{-1}y.
\]
Since \(C^{xG}(X)=C(X)G^T\), the particle positions satisfy
\begin{equation}
	\begin{aligned}
		\ddot x_j(t)
		+
		\gamma \dot x_j(t)
		&=
		-\beta C(X(t))(Ax_j(t)-b) \\
		&\quad
		+
		k\sum_{i\neq j}
		f(\|x_i(t)-x_j(t)\|)
		(x_j(t)-x_i(t)) \\
		&\quad
		-
		\alpha(x_j(t)-\bar x(t)),
	\end{aligned}
	\label{eq:second-order-position-form}
\end{equation}
for \(j=1,\dots,J\). It is a covariance-preconditioned heavy-ball dynamics coupled with attraction--repulsion interactions within the ensemble. This is the form used throughout the analysis. The associated velocities are \(v_j=\dot x_j\).

We first record the equations satisfied by the ensemble mean and the fluctuations. Let
\[
\tilde x_j(t):=x_j(t)-\bar x(t),
\qquad
\frac1J\sum_{j=1}^J \tilde x_j(t)=0.
\]

\begin{lemma}[Mean and fluctuation dynamics]
	\label{lem:second-order-mean-fluctuations}
	Let \(X(t)\) be the position component of a solution of the linear second-order EKI system \eqref{eq:second-order-position-form}. Then the ensemble mean satisfies
	\begin{equation}
		\ddot{\bar x}(t)
		+
		\gamma \dot{\bar x}(t)
		=
		-\beta C(X(t))(A\bar x(t)-b),
		\label{eq:second-order-mean-dynamics}
	\end{equation}
	whereas the fluctuations satisfy
	\begin{equation}
		\begin{aligned}
			\ddot{\tilde x}_j(t)
			+
			\gamma \dot{\tilde x}_j(t)
			&=
			-\beta C(X(t))A\tilde x_j(t)
			-
			\alpha \tilde x_j(t)
			+
			k\sum_{i\neq j}
			f(\|\tilde x_i(t)-\tilde x_j(t)\|)
			\bigl(\tilde x_j(t)-\tilde x_i(t)\bigr),
		\end{aligned}
		\label{eq:second-order-fluctuation-dynamics}
	\end{equation}
	for \(j=1,\dots,J\).
\end{lemma}

\begin{proof}
Summing~\eqref{eq:second-order-position-form} over $j$ and dividing by $J$ gives~\eqref{eq:second-order-mean-dynamics}. Indeed, the attraction term averages to zero and the repulsion also vanishes by antisymmetry. Subtracting~\eqref{eq:second-order-mean-dynamics} from~\eqref{eq:second-order-position-form} gives~\eqref{eq:second-order-fluctuation-dynamics}.
\end{proof}

The mean dynamics~\eqref{eq:second-order-mean-dynamics} is an inertial EKI equation. The internal ensemble geometry is controlled by~\eqref{eq:second-order-fluctuation-dynamics}, where Kalman contraction, mean attraction, and pairwise repulsion compete. This separation is one of the useful structural properties of the model.

\begin{proposition}[Local well-posedness and continuation criterion]
	\label{prop:local-well-posedness}
	Assume that \(f\in C^1([0,\infty))\) and that
	\(z\mapsto f(\|z\|)z\) is locally Lipschitz on \(\mathbb R^d\). Then, for every initial datum
	\[
	(X(0),\dot X(0))\in (\mathbb{R}^d)^J\times(\mathbb{R}^d)^J,
	\]
	the system \eqref{eq:second-order-position-form} admits a unique maximal classical solution
	\[
	(X,\dot X)\in C^1([0,T_{\max});(\mathbb{R}^d)^J\times(\mathbb{R}^d)^J),
	\]
	with \(0<T_{\max}\leq+\infty\). Moreover, if
	\[
	\sup_{t\in[0,T_{\max})}
	\left(
	\sum_{j=1}^J \|x_j(t)\|^2
	+
	\sum_{j=1}^J \|\dot x_j(t)\|^2
	\right)
	<+\infty,
	\]
	then \(T_{\max}=+\infty\).
\end{proposition}

\begin{proof}
	Writing the system in first-order form for the variables
	\[
	Z(t)=(X(t),\dot X(t))\in(\mathbb{R}^d)^J\times(\mathbb{R}^d)^J,
	\]
	we obtain an autonomous ODE
	\[
	\dot Z(t)=F(Z(t)).
	\]
	The damping and attraction terms are linear. The Kalman term
	\[
	-\beta C(X)(Ax_j-b)
	\]
	is polynomial in the particle positions, because \(C(X)\) is quadratic in \(X\). Hence it is locally Lipschitz. The repulsive term is locally Lipschitz by the assumption on \(z\mapsto f(\|z\|)z\). Therefore \(F\) is locally Lipschitz on the phase space. The Picard--Lindelöf theorem gives existence and uniqueness of a maximal solution on some interval \([0,T_{\max})\).
	
	The continuation criterion is the standard one for finite-dimensional ODEs with locally Lipschitz vector fields. If the solution remains bounded as \(t\uparrow T_{\max}\), then \(F\) remains bounded and locally Lipschitz on a compact set containing the trajectory.
\end{proof}

For the regularized inverse-power kernel
\eqref{eq:repulsive-kernel}
the interaction term is nonsingular at particle collisions and satisfies the local Lipschitz assumption in Proposition~\ref{prop:local-well-posedness}. Thus local well-posedness holds without any additional structural assumption.

\begin{remark}[On global-in-time bounds]
	\label{rem:global-well-posedness-commutation}
	A possible route to global-in-time bounds is to impose the commutation property
	\[
	AC(X(t))=C(X(t))A
	\qquad\text{for all } t\geq0.
	\]
	Under this additional assumption, the Kalman term in the fluctuation equation is compatible with a Lyapunov functional. Let \(W\) be a potential associated with the repulsive force, defined by
	\[
	W'(r)=-r f(r).
	\]
	Consider the internal energy
	\[
	\begin{aligned}
		E_{\rm int}(t)
		&=
		\frac12\sum_{j=1}^J \|\dot{\tilde x}_j(t)\|^2
		+
		\frac{\alpha}{2}\sum_{j=1}^J \|\tilde x_j(t)\|^2 \\
		&\quad
		+
		\frac{k}{2}\sum_{i\neq j}
		W(\|\tilde x_i(t)-\tilde x_j(t)\|)
		+
		\frac{\beta J}{4}
		\operatorname{Tr}\!\left(C(X(t))A C(X(t))\right).
	\end{aligned}
	\]
	Using the fluctuation equation \eqref{eq:second-order-fluctuation-dynamics}, the attraction and repulsion terms cancel with the derivatives of their corresponding potentials. The remaining Kalman contribution cancels with the derivative of the last term
	provided \(A\) and \(C(X(t))\) commute. In that case
	\[
	\frac{\mathrm{d}}{\mathrm{d}t}E_{\rm int}(t)
	=
	-\gamma\sum_{j=1}^J \|\dot{\tilde x}_j(t)\|^2
	\leq0.
	\]
	Since the attraction term is quadratic and dominates the repulsive potential at large distances, this estimate yields uniform bounds on the fluctuations and on their velocities. In particular, the covariance matrix \(C(X(t))\), which depends only on the fluctuations, remains bounded and, thus, the mean then solves the non-autonomous linear equation
	\[
	\ddot{\bar x}(t)
	+
	\gamma\dot{\bar x}(t)
	+
	\beta C(X(t))A\bar x(t)
	=
	\beta C(X(t))b,
	\]
	with bounded coefficients on finite time intervals. Hence the mean cannot blow up in finite time.
\end{remark}

\subsection{Instability of collapsed configurations and non-collapsed equilibria}
\label{subsec:collapse-instability}

We now study the local stability of fully collapsed configurations. These configurations are stationary states of the particle system. If all particles have the same position and zero velocity, the empirical covariance vanishes and both the attraction and repulsion terms are zero. The following result shows that, for a sufficiently strong repulsive interaction, such configurations are unstable with respect to zero-mean perturbations.

\begin{proposition}[Instability of collapsed configurations]
	\label{prop:collapsed-instability}
	Assume that \(f\) is continuous at the origin and let \(f(0)>0\). Consider a fully collapsed configuration
	\[
	x_1=\cdots=x_J=x_*,
	\qquad
	\dot x_1=\cdots=\dot x_J=0.
	\]
	Then this configuration is linearly unstable with respect to zero-mean fluctuation perturbations under the supercritical condition
	\begin{equation}
		kJf(0)>\alpha.
		\label{eq:collapse-instability-condition}
	\end{equation}
	In particular, for the kernel \eqref{eq:repulsive-kernel}, condition \eqref{eq:collapse-instability-condition} becomes
	\begin{equation*}
		k>\frac{\alpha\varepsilon^p}{J}.
		%\label{eq:collapse-instability-condition-kernel}
	\end{equation*}
\end{proposition}

\begin{proof}
	Let
	\[
	x_j(t)=x_*+\delta x_j(t),
	\qquad
	\dot x_j(t)=\delta \dot x_j(t),
	\]
	where the perturbations are assumed to have zero mean.
	At the fully collapsed configuration the empirical covariance is zero. Therefore the linearization of the Kalman term does not contribute to first order.
	The attraction term linearizes as
	\[
	-\alpha(x_j-\bar x)
	=
	-\alpha \delta x_j.
	\]
	For the repulsive term, since \(f\) is continuous at the origin, we have to first order
	\[
	f(\|x_i-x_j\|)=f(0)+o(1).
	\]
	Thus, using the zero-mean condition,
	\begin{align*}
	k\sum_{i\neq j}
	f(\|x_i-x_j\|)
	(x_j-x_i)
	&=
	kf(0)\sum_{i\neq j}(\delta x_j-\delta x_i)
	+
	\text{higher-order terms} \\
	&= kf(0) J\delta x_j
	+
	\text{higher-order terms}
	\end{align*}
	
	Hence, using \eqref{eq:second-order-fluctuation-dynamics}, we obtain that the linearized fluctuation dynamics is
	\begin{equation*}
	\ddot{\delta x}_j
	+
	\gamma\dot{\delta x}_j
	-
	\bigl(kJf(0)-\alpha\bigr)\delta x_j
	=
	0, \qquad j=1,\dots,J.
	%\label{eq:linearized-collapsed-fluctuations}
	\end{equation*}
	The characteristic equation for each fluctuation mode is
	\[
	\lambda^2+\gamma\lambda-\bigl(kJf(0)-\alpha\bigr)=0.
	\]
	If \(kJf(0)>\alpha\), the product of the two roots is negative. Therefore one eigenvalue is positive and the collapsed configuration is linearly unstable in the zero-mean fluctuation subspace.
	
	For \(f(r)=(\varepsilon+r)^{-p}\), one has \(f(0)=\varepsilon^{-p}\), and the condition \(kJf(0)>\alpha\) is equivalent to
	\[
	k>\frac{\alpha\varepsilon^p}{J}.
	\]
	This concludes the proof.
\end{proof}

\begin{remark}[Neutral translation mode]
	The instability described in Proposition~\ref{prop:collapsed-instability} is transverse to the collective translation mode. Perturbations where all particles are shifted by the same vector do not change the ensemble spread and they are not described by the zero-mean fluctuation dynamics. The relevant stability property for covariance collapse is therefore stability with respect to perturbations of the internal configuration of the ensemble.
\end{remark}

\begin{remark}[Subcritical regime]
	If the subcritical condition \(kJf(0)<\alpha\) holds true, collapsed configurations are linearly stable with respect to zero-mean fluctuation perturbations. Therefore the attraction dominates the linearized repulsion which is too weak to destabilize covariance collapse.
\end{remark}

We relate the instability of collapsed configurations to the possible asymptotic states of the ensemble.

Recall that an ensemble configuration \(X\) is collapsed if \(C(X)=0\).

The ensemble spread is given by
\[
S(X(t)):=\frac1J\sum_{j=1}^J\|\tilde x_j(t)\|^2
=\operatorname{Tr}(C(X(t))),
\]
and
\begin{align}
	\frac{\mathrm d^2}{\mathrm dt^2}S(X(t))
	+
	\gamma\frac{\mathrm d}{\mathrm dt}S(X(t))
	&=
	\frac{2}{J}\sum_{j=1}^J
	\|\dot{\tilde x}_j(t)\|^2
	-
	2\beta\operatorname{Tr}\!\left(C(X(t))AC(X(t))\right)
	\notag\\
	&\quad
	-
	2\alpha S(X(t))
	+
	\frac{k}{J}
	\sum_{i\neq j}
	f(r_{ij}(t))r_{ij}(t)^2,
	\label{eq:spread-dynamics}
\end{align}
where \(r_{ij}(t)=\|\tilde x_i(t)-\tilde x_j(t)\|\). This identity shows the competition between the repulsive contribution, which increases the spread, and the attraction and Kalman terms, that tend to decrease it.

\begin{corollary}[Non-collapsed attracting equilibria]
	\label{prop:noncollapsed-asymptotic-equilibria}
	Assume that \(f\) is continuous at the origin, \(f(0)>0\), and
	\eqref{eq:collapse-instability-condition} holds true.
	Let \((X(t),\dot X(t))\) be a global solution of \eqref{eq:second-order-position-form} converging to an asymptotic equilibrium
	\[
	(X(t),\dot X(t))\to (X_\infty,0)
	\qquad\text{as } t\to\infty.
	\]
	If the limiting equilibrium is asymptotically stable with respect to zero-mean fluctuation perturbations, then it cannot be fully collapsed, i.e.
	\[
	C(X_\infty)\neq0 \quad \text{ and } \quad S(X_\infty)>0.
	\]
\end{corollary}

\begin{proof}
	Assume by contradiction that the limiting equilibrium is fully collapsed. Then
	\[
	x_1^\infty=\cdots=x_J^\infty=x_*,
	\qquad
	\dot x_1^\infty=\cdots=\dot x_J^\infty=0,
	\]
	that is \(C(X_\infty)=0\).
	
	However, by Proposition~\ref{prop:collapsed-instability}, under condition
	\eqref{eq:collapse-instability-condition}, every fully collapsed equilibrium is linearly unstable with respect to zero-mean fluctuation perturbations. Hence such an equilibrium cannot be asymptotically stable in the internal configuration of the ensemble. This contradicts the assumed asymptotic stability of the limiting equilibrium.
\end{proof}

%{\color{black}
%	\begin{remark} \label{rem:repulsion:independent:J}
%		Under the normalized parametrization of Remark \ref{rem:normalization:repulsion}, the collapse-instability
%		condition \(kJf(0)>\alpha\) reduces to
%		\[
%		\kappa f(0)>\alpha,
%		\]
%		and is therefore independent of the ensemble size. The effect of this normalization is investigated numerically in Section~\ref{sec:darcy-ensemble-size}.
%	\end{remark}
%}

{\color{black}
	
	\begin{remark}[Role and selection of the parameters]
		\label{rem:parameter-choice}\label{rem:repulsion:independent:J}
		The parameters of the second-order dynamics have distinct roles.
		The coefficients $\gamma$ and $\beta$ control, respectively, the damping
		of the inertial motion and the strength of the Kalman force, whereas
		$\alpha$ and $k$ determine the competition between attraction and
		repulsion in the fluctuation dynamics. For the inverse-power kernel
		\eqref{eq:repulsive-kernel}, $\varepsilon$ regularizes the interaction at short distances
		and $p$ controls its spatial decay.
		
		The instability condition
		\[
		kJf(0)>\alpha
		\]
		provides a structural criterion for whether complete ensemble collapse is
		locally unstable, but it should not be interpreted as a general parameter
		selection rule. In particular, increasing the repulsion strength beyond
		this threshold does not necessarily improve the inversion: excessive
		repulsion may lead to an overdispersed ensemble. Conversely, dominant
		attraction may promote contraction. This balance is illustrated
		numerically by the interaction ablation in
		Section~\ref{sec:darcy-interaction-ablation}.
		
		When varying the ensemble size, the normalization
		$k=k_J=\kappa/J$ of Remark~\ref{rem:normalization:repulsion} removes the explicit linear dependence
		on $J$ from the local collapse-instability condition. Beyond these
		structural considerations, however, the appropriate quantitative balance
		between the parameters remains problem dependent, and we do not propose
		a universal tuning strategy in the present work.
	\end{remark}
	
}

\subsection{Limiting mean dynamics: optimality and frozen-covariance decay}
\label{subsec:limiting-mean-dynamics}

We now discuss the behavior of the ensemble mean in the limiting covariance regime. 
First, we characterize the optimality condition satisfied by asymptotic equilibria. 
Then we consider the frozen-covariance dynamics, which provides a linear asymptotic model along the retained directions.

The limiting covariance may be nonzero but still rank-deficient. Therefore, the equilibrium condition for the ensemble mean does not necessarily imply full optimality in \(\mathbb{R}^d\). It only yields optimality along the directions retained by the limiting ensemble covariance.

\begin{proposition}[Optimality on the retained subspace]
	\label{prop:optimality-retained-subspace}
	Let \((X(t),\dot X(t))\) be a global solution of \eqref{eq:second-order-position-form} such that
	\[
	(X(t),\dot X(t))\to (X_\infty,0)
	\qquad\text{as } t\to\infty.
	\]
	Let
	\[
	\bar x_\infty
	:=
	\frac1J\sum_{j=1}^J x_j^\infty,
	\qquad
	C_\infty:=C(X_\infty),
	\]
	and define the retained subspace
	$
	\mathcal V_\infty:=\operatorname{Range}(C_\infty).
	$
	Then
	\begin{equation}
		C_\infty(A\bar x_\infty-b)=0,
		\label{eq:retained-optimality-condition}
	\end{equation}
	i.e. \(\bar x_\infty\) is a stationary point of \(\Phi\) restricted to
	$
	\mathcal A_\infty:=\bar x_\infty+\mathcal V_\infty.
	$
	If, in addition, \(A\) is positive definite on \(\mathcal V_\infty\), then \(\bar x_\infty\) is the unique minimizer of \(\Phi\) on \(\mathcal A_\infty\).
\end{proposition}

\begin{proof}
	By the mean equation \eqref{eq:second-order-mean-dynamics},
	\[
	\ddot{\bar x}(t)
	+
	\gamma\dot{\bar x}(t)
	=
	-\beta C(X(t))(A\bar x(t)-b).
	\]
	Since \((X(t),\dot X(t))\to (X_\infty,0)\), we have
	\[
	\dot{\bar x}(t)\to0,
	\qquad
	C(X(t))\to C_\infty,
	\qquad
	\bar x(t)\to\bar x_\infty.
	\]
	At the limiting equilibrium, the acceleration of the mean vanishes. Passing to the limit in the mean equation gives
	\[
	C_\infty(A\bar x_\infty-b)=0.
	\]
	
	Since \(C_\infty\) is symmetric positive semidefinite, its range is orthogonal to its null space:
	\[
	\operatorname{Range}(C_\infty)
	=
	\operatorname{Ker}(C_\infty)^\perp.
	\]
	Therefore \eqref{eq:retained-optimality-condition} implies
	$
	A\bar x_\infty-b
	=
	\nabla\Phi(\bar x_\infty)
	\in
	\operatorname{Ker}(C_\infty)
	=
	\mathcal V_\infty^\perp.
	$
	Hence,
	$
	\nabla\Phi(\bar x_\infty)\perp\mathcal V_\infty,
	$
	and \(\Phi(\bar x_\infty)\) is first-order optimal on the affine space \(\mathcal A_\infty=\bar x_\infty+\mathcal V_\infty\).
	
	Finally, if \(A\) is positive definite on \(\mathcal V_\infty\), then the restriction of the quadratic functional \(\Phi\) to \(\mathcal A_\infty\) is strictly convex. Therefore the stationary point \(\bar x_\infty\) is the unique minimizer of \(\Phi\) on that affine space.
\end{proof}

We conclude the analysis with a decay estimate for the mean dynamics in the frozen-covariance regime.

Let
\[
\mathcal V_\infty:=\operatorname{Range}(C_\infty)
\]
be the subspace retained by the limiting covariance, and let \(x^\dagger\) be a minimizer of \(\Phi\) on the affine space associated with \(\mathcal V_\infty\). Equivalently, \(x^\dagger\) satisfies
\begin{equation}
	C_\infty(Ax^\dagger-b)=0.
	\label{eq:frozen-retained-minimizer}
\end{equation}
The frozen-covariance mean dynamics is
\begin{equation}
	\ddot{\bar x}(t)
	+
	\gamma \dot{\bar x}(t)
	=
	-\beta C_\infty(A\bar x(t)-b).
	\label{eq:frozen-mean-dynamics}
\end{equation}

\begin{proposition}[Decay for the frozen-covariance mean dynamics]
	\label{prop:frozen-covariance-decay}
	Assume that \(A\) is positive definite on \(\mathcal V_\infty\). Let \(x^\dagger\) satisfy \eqref{eq:frozen-retained-minimizer}, and assume that the initial error and velocity of the frozen dynamics satisfy
	\[
	\bar x(0)-x^\dagger\in\mathcal V_\infty,
	\qquad
	\dot{\bar x}(0)\in\mathcal V_\infty.
	\]
	Then the error
	\[
	r(t):=\bar x(t)-x^\dagger
	\]
	decays exponentially to zero along the retained directions.
\end{proposition}

\begin{proof}
	Subtracting \eqref{eq:frozen-retained-minimizer} from \eqref{eq:frozen-mean-dynamics}, we obtain
	\begin{equation}
		\ddot r(t)
		+
		\gamma \dot r(t)
		+
		\beta C_\infty A r(t)
		=
		0.
		\label{eq:frozen-error-dynamics}
	\end{equation}
	Since \(C_\infty\) is symmetric positive semidefinite, its restriction to \(\mathcal V_\infty\) is positive definite. Moreover, by assumption, \(A\) is positive definite on \(\mathcal V_\infty\).
	
	On \(\mathcal V_\infty\), introduce the symmetric positive definite matrix
	\[
	B:=C_\infty^{1/2} A C_\infty^{1/2}.
	\]
	Setting
	\[
	r(t)=C_\infty^{1/2} z(t),
	\]
	equation \eqref{eq:frozen-error-dynamics} is equivalent, on the retained subspace, to
	\begin{equation*}
		\ddot z(t)
		+
		\gamma \dot z(t)
		+
		\beta B z(t)
		=
		0.
		%\label{eq:frozen-z-dynamics}
	\end{equation*}
	Since \(B\) is symmetric positive definite on \(\mathcal V_\infty\), it admits an orthonormal basis of eigenvectors with eigenvalues
	\[
	0<\mu_1\leq \cdots \leq \mu_m,
	\qquad
	m=\dim\mathcal V_\infty.
	\]
	Expanding
	\[
	z(t)=\sum_{\ell=1}^m q_\ell(t)z_\ell,
	\]
	we obtain the decoupled scalar equations
	\[
	\ddot q_\ell(t)
	+
	\gamma \dot q_\ell(t)
	+
	\beta \mu_\ell q_\ell(t)
	=
	0,
	\qquad
	\ell=1,\dots,m.
	\]
	The characteristic roots are
	\[
	\lambda_\ell^\pm
	=
	\frac{-\gamma\pm\sqrt{\gamma^2-4\beta\mu_\ell}}{2}.
	\]
	Since \(\gamma>0\), \(\beta>0\), and \(\mu_\ell>0\), both roots have strictly negative real part. Hence each modal component \(q_\ell(t)\) decays exponentially to zero. Therefore \(z(t)\), and consequently \(r(t)=C_\infty^{1/2}z(t)\), decay exponentially along the retained directions.
\end{proof}

\section{Numerical experiments}
\label{sec:numerics}

{\color{black}
	In this section we investigate the proposed second-order EKI on
	representative inverse problems, focusing on accessible-subspace
	enlargement, discrepancy-based regularization, interaction effects on
	ensemble geometry, and dependence on the ensemble size, with first-order
	EKI-type baselines where appropriate.
}

{\color{black}
	
	\subsection{Numerical setup, diagnostics, and stopping criteria}
	\label{sec:numerical-method}
	\label{subsec:numerical-time-discretization}
	\label{subsec:numerical-diagnostics}
	
	\paragraph{Time discretization.}
	Unless stated otherwise, we discretize the second-order EKI system
	\eqref{eq:second-order-eki} by an implicit--explicit scheme
	\cite{PareschiRusso2005}. The position equation is advanced explicitly,
	while the velocity equation is treated implicitly only in the linear
	damping term. Given
	\[
	X^n=\{x_j^n\}_{j=1}^J,
	\qquad
	V^n=\{v_j^n\}_{j=1}^J,
	\]
	we first set
	\begin{equation}
		x_j^{n+1}
		=
		x_j^n+\Delta t\,v_j^n,
		\qquad
		j=1,\ldots,J.
		\label{eq:numerics-position-update}
	\end{equation}
	The velocity is then updated according to
%	\begin{align}
%		v_j^{n+1}
%		&=
%		v_j^n
%		+
%		\Delta t
%		\Big[
%		-\gamma v_j^{n+1}
%		+
%		\beta C^{xG}(X^{n+1})\Gamma^{-1}
%		\bigl(y-G(x_j^{n+1})\bigr)
%		\Big]
%		\notag\\
%		&\quad
%		+
%		\Delta t
%		\Bigg[
%		k\sum_{i\neq j}
%		f(\|x_i^n-x_j^n\|)
%		(x_j^n-x_i^n)
%		-
%		\alpha(x_j^n-\bar x^n)
%		\Bigg].
%		\label{eq:numerics-velocity-imex-implicit}
%	\end{align}
%	Equivalently,
	\begin{align}
		v_j^{n+1}
		&=
		\frac{1}{1+\gamma\Delta t}v_j^n
		+
		\frac{\beta\Delta t}{1+\gamma\Delta t}
		C^{xG}(X^{n+1})\Gamma^{-1}
		\bigl(y-G(x_j^{n+1})\bigr)
		\notag\\
		&\quad
		+
		\frac{k\Delta t}{1+\gamma\Delta t}
		\sum_{i\neq j}
		f(\|x_i^n-x_j^n\|)
		(x_j^n-x_i^n)
		-
		\frac{\alpha\Delta t}{1+\gamma\Delta t}
		(x_j^n-\bar x^n).
		\label{eq:numerics-velocity-imex}
	\end{align}
	No nonlinear system has to be solved at each time step.
	
	The scheme
	\eqref{eq:numerics-position-update}--\eqref{eq:numerics-velocity-imex}
	is used for the second-order dynamics throughout the numerical section
	unless a different discretization is explicitly stated. In particular,
	the noise-dependence experiment below uses a semi-implicit variant in
	order to handle the increasing stiffness of the Kalman force as the
	observational noise is reduced.
	
	\paragraph{Data misfit and reconstruction error.}
	We monitor the weighted data misfit at the ensemble mean,
	$\Phi(\bar x^n)$.
	
	Throughout the numerical experiments, $\|\cdot\|_h$ denotes the discrete
	norm associated with the parameter representation. For vector-valued
	tests we use the normalized Euclidean norm
	$
	\|z\|_h=d^{-1/2}\|z\|_2,
	$
	whereas for spatially discretized PDE parameters we use the corresponding
	discrete $L^2$ norm.
	
	In synthetic experiments, where the reference parameter
	\(x^\dagger\) is available, we also report the relative reconstruction
	error
	\begin{equation*}
		e_x^n
		=
		\frac{\|\bar x^n-x^\dagger\|_h}
		{\|x^\dagger\|_h}.
		%\label{eq:numerical-mean-error}
	\end{equation*}
	This quantity is used exclusively as an a posteriori diagnostic. It does
	not enter the inversion dynamics, the choice of the algorithmic
	parameters, or the stopping criterion.
	
	\paragraph{Ensemble spread and effective rank.}
	To quantify ensemble diversity, we use the root-mean-square spread
	\begin{equation*}
		S^n :=
		\left(
		\frac1J
		\sum_{j=1}^J
		\|x_j^n-\bar x^n\|_h^2
		\right)^{1/2}.
		%\label{eq:numerical-spread}
	\end{equation*}
	
	The spread measures the overall magnitude of the ensemble fluctuations
	but does not distinguish whether this variance is distributed over many
	directions or concentrated in a lower-dimensional subset. We therefore
	also use the entropy effective rank of the empirical covariance \(C(X^n)\).
	Let
%	\[
%	C(X^n)
%	=
%	\frac1J
%	\sum_{j=1}^J
%	(x_j^n-\bar x^n)
%	\otimes
%	(x_j^n-\bar x^n),
%	\]
%	and let
	\(\lambda_1,\ldots,\lambda_r>0\) denote its nonzero eigenvalues. Defining
	\[
	\pi_i
	=
	\frac{\lambda_i}
	{\sum_{\ell=1}^r\lambda_\ell},
	\]
	we set
	\begin{equation*}
		r_{\rm eff}(X^n)
		=
		\exp\left(
		-\sum_{i=1}^r
		\pi_i\log\pi_i
		\right).
		%\label{eq:numerical-effective-rank}
	\end{equation*}
	%The effective rank measures the number of directions carrying a
	%significant fraction of the ensemble variance and complements the scalar
	%spread diagnostic when assessing covariance collapse.
	If \(C(X^n)\) is numerically indistinguishable from zero, corresponding
	to numerical ensemble collapse, we set \(r_{\rm eff}(X^n)=0\).
	
	\paragraph{Discrepancy stopping.}
	For additive Gaussian observational noise
	$
	\eta\sim\mathcal N(0,\Gamma),
	$
	the whitened noise energy satisfies
	\[
	\|\eta\|_\Gamma^2
	=
	\eta^T\Gamma^{-1}\eta
	\sim\chi_K^2.
	\]
	We therefore use discrepancy-type early stopping of the form
	\begin{equation}
		\Phi(\bar x^n)
		\leq
		\Phi_{\rm disc},
		\qquad
		\Phi_{\rm disc}
		=
		\frac{\tau K}{2},
		\label{eq:discrepancy-stopping}
	\end{equation}
	where the tolerance \(\tau\) is fixed independently of the unknown
	parameter.
	
	For the linear inverse problems below we use
	\[
	\tau=1.05.
	\]
	For the nonlinear Darcy experiments, where \(K=64\), we instead use the
	a priori \(95\%\) Gaussian confidence threshold
	\begin{equation*}
		\Phi_{\rm disc}
		=
		\frac12\chi^2_{K,0.95},
		%\label{eq:chi-square-discrepancy}
	\end{equation*}
	which is equivalent to
	\[
	\tau
	=
	\frac{\chi^2_{K,0.95}}{K}
	\simeq1.3074.
	\]
	The two choices play the same role: they prescribe the admissible data
	residual from the observational noise model without using the true
	parameter.
	
	The reported discrepancy stopping time is the
	first discrete time at which
	\eqref{eq:discrepancy-stopping} is satisfied. A finite final integration
	time is used as a numerical search horizon. If the discrepancy threshold
	is not reached within that interval, the run is reported as unsuccessful
	with respect to discrepancy stopping, and terminal or minimum-over-time
	diagnostics are reported when relevant.
	
	\paragraph{Computational work.}
	When comparing computational effort, each evaluation of the forward map
	\(G\), either at an ensemble member or at the ensemble mean for diagnostic
	purposes, is counted as one forward solve. Forward-solve counts provide a
	problem-independent proxy for the dominant inverse-problem cost, while
	the direct evaluation of the pairwise interaction requires additional
	algebraic work of order \(J^2\). Wall-clock times are therefore reported
	only as complementary implementation-dependent diagnostics. All numerical
	codes are implemented in MATLAB; the runtime measurements reported below
	were obtained using MATLAB Online.
	
}

{\color{black}
	
	\subsection{Numerical verification of accessible-subspace enlargement}
	\label{sec:numerics-subspace}
	
	%We first consider a controlled linear inverse problem designed to
	%illustrate the enlargement of the accessible subspace induced by
	%nonzero initial velocities.
	
	We consider the one-dimensional elliptic problem
	\begin{equation*}
		-p''(s)+p(s)=u(s),
		\qquad s\in(0,\pi),
		\qquad
		p(0)=p(\pi)=0,
		%\label{eq:elliptic-subspace}
	\end{equation*}
	and recover the forcing term $u$ from noisy observations of the state $p$.
	After a finite-difference discretization on $d$ interior grid points, the
	forward map is linear and can be written as
	\[
	p=Gu,
	\qquad G\in\mathbb R^{d\times d}.
	\]
	Synthetic observations are generated according to
	\[
	y=Gu^\dagger+\eta,
	\qquad
	\eta\sim\mathcal N(0,\sigma^2 I),
	\]
	with
	\[
	d=K=255,
	\qquad
	J=20,
	\qquad
	\sigma=0.1,
	\qquad
	u^\dagger(s)=10\sin(8s).
	\]
	Thus the experiment is performed in the finite-ensemble regime $J\ll d$.
	
	Let $S_\ell$ denote the normalized discrete representation of
	$\sin(\ell s)$. The initial position ensemble is generated from the first six
	sine modes,
	\[
	x_j^0
	=
	\sigma_0\sum_{\ell=1}^{6}\xi_{\ell,j}S_\ell,
	\qquad
	\sigma_0=5\times10^{-2},
	\]
	where the random coefficients are centered over the ensemble. Consequently,
	$\bar x^0=0$ up to round-off and
	\[
	\mathcal S_x
	:=
	\operatorname{span}\{x_j^0-\bar x^0:j=1,\ldots,J\}
	\]
	has dimension six. Since the true parameter is proportional to the eighth
	sine mode, $u^\dagger\notin\mathcal S_x$.
	
	Proposition~\ref{prop:generalized-subspace} shows that, for arbitrary initial
	velocities, the second-order trajectories remain in the affine space generated
	jointly by the initial position anomalies and velocities. To verify this numerically, we consider
	\[
	v_j^0=0,
	\qquad
	v_j^0\in\mathcal S_x,
	\qquad
	v_j^0\in\mathcal S_x^\perp.
	\]
	For the last two cases the velocities are centered and normalized so that
	\[
	\left(
	\frac1J\sum_{j=1}^J\|v_j^0\|_h^2
	\right)^{1/2}
	=0.5.
	\]
	The in-span velocities are generated within $\mathcal S_x$, whereas the
	out-of-span velocities are generated from the modes $7,\ldots,12$. Hence,
	in the latter case,
	\[
	\mathcal S_{xv}
	:=
	\operatorname{span}
	\{x_j^0-\bar x^0,v_j^0:j=1,\ldots,J\}
	\]
	has dimension twelve and contains the direction of $u^\dagger$.
	
	For this structural verification we use
	\[
	\gamma=2,
	\qquad
	\beta=0.5,
	\qquad
	\alpha=0.4,
	\qquad
	k=0.01,
	\qquad
	\varepsilon=1,
	\qquad
	p=1.5,
	\]
	and $\Delta t=5\cdot10^{-2}$. These parameters satisfy
	\[
	kJf(0)=0.2<\alpha,
	\]
	so that transverse round-off perturbations are not amplified by the
	collapse-instability mechanism.
	
	Figure~\ref{fig:subspace-verification} reports the corresponding projection
	residuals. For each particle, the distance from the accessible affine space
	is measured through
	\[
	\|(I-P_{\mathcal S})(x_j(t)-\bar x(0))\|,
	\]
	with $P_{\mathcal S}$ the orthogonal projector onto $\mathcal S$; the reported
	quantity is the corresponding ensemble norm. With zero or in-span initial
	velocities, the distance from $\mathcal S_x$ remains of order $10^{-15}$,
	whereas out-of-span velocities leave $\mathcal S_x$ while remaining in
	$\mathcal S_{xv}$ up to round-off.
	
	\begin{figure}[t]
		\centering
		\includegraphics[width=0.92\textwidth]
		{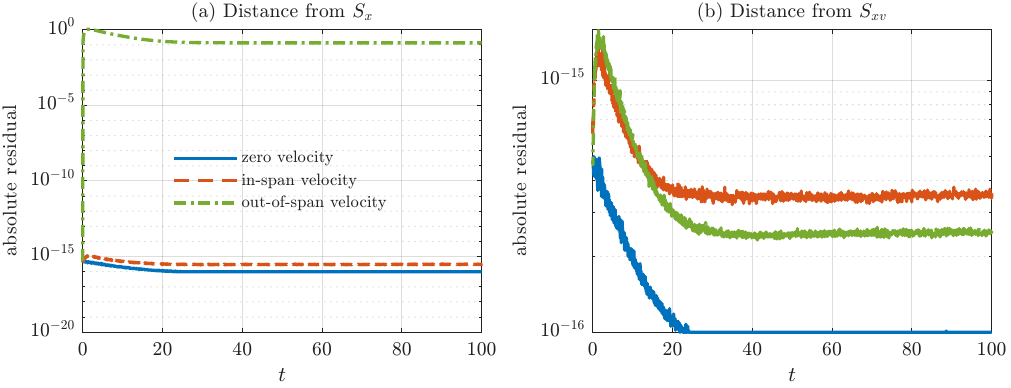}
		\caption{Numerical verification of the generalized subspace property.
			Left: distance of the particle positions from the initial position space
			$\mathcal S_x$. Right: distance from the generalized position--velocity
			space $\mathcal S_{xv}$. Zero and in-span initial velocities preserve
			$\mathcal S_x$, whereas out-of-span initial velocities leave
			$\mathcal S_x$ while remaining in $\mathcal S_{xv}$ up to numerical
			round-off.}
		\label{fig:subspace-verification}
	\end{figure}
	
	The effect of this enlargement can also be quantified independently of the
	particle dynamics. For a linear subspace $\mathcal S$, define
	\[
	e^\star(\mathcal S)
	:=
	\min_{x\in\bar x^0+\mathcal S}
	\frac{\|x-u^\dagger\|_h}{\|u^\dagger\|_h}
	\quad \
	\text{and}
	\ \quad
	\Phi^\star(\mathcal S)
	:=
	\min_{x\in\bar x^0+\mathcal S}\Phi(x).
	\]
%	Both quantities are computed by restricting the problem to
%	$\bar x^0+\mathcal S$: $e^\star(\mathcal S)$ by orthogonal projection,
%	and $\Phi^\star(\mathcal S)$ by solving the corresponding restricted
%	linear least-squares problem.
	Table~\ref{tab:subspace-geometry} compares
	these quantities for $\mathcal S_x$ and $\mathcal S_{xv}$. Here
	$\Phi_{\rm disc}=133.875$.
	
	The position-only space satisfies
	\[
	\Phi^\star(\mathcal S_x)=264.05>\Phi_{\rm disc},
	\]
	so that no trajectory confined to $\bar x^0+\mathcal S_x$ can satisfy the
	discrepancy principle. Hence, standard EKI initialized from the same position
	ensemble $X^0$, as well as any span-preserving inflation that preserves
	the initial anomaly span, remains confined to $\bar x^0+\mathcal S_x$ and
	cannot satisfy the discrepancy principle in this experiment. By contrast,
	\[
	\Phi^\star(\mathcal S_{xv})=102.31<\Phi_{\rm disc},
	\]
	and the true parameter itself is represented in the enlarged space up to
	round-off.
	
	\begin{table}[t]
		\centering
		\caption{Geometry of the accessible affine spaces for the elliptic inverse
			problem. Standard EKI and span-preserving inflation initialized
			from $X^0$ remain confined to $\bar x^0+\mathcal S_x$, whereas
			out-of-span initial velocities enlarge the second-order accessible space
			to $\bar x^0+\mathcal S_{xv}$.}
		\label{tab:subspace-geometry}
		\small
		\begin{tabular}{lccc}
			\toprule
			Accessible space / associated dynamics
			& Dimension & $e^\star$ & $\Phi^\star$ \\
			\midrule
			$\mathcal S_x$ (standard EKI / span-preserving inflation)
			& 6 & 1.0000 & 264.05 \\
			$\mathcal S_{xv}$ (second-order, out-of-span $V_0$)
			& 12 & $1.23\cdot10^{-15}$ & 102.31 \\
			\bottomrule
		\end{tabular}
	\end{table}
	
%	Thus the initial velocity directions do more than introduce kinetic energy:
%	they enlarge the affine space available to the second-order dynamics.
%	In the
%	present construction, the position-only space cannot contain a
%	discrepancy-compatible solution, whereas the enlarged position--velocity
%	space can and contains the true parameter up to round-off. This controlled
%	experiment therefore provides a direct numerical counterpart of
%	Proposition~\ref{prop:generalized-subspace}.
%	The extent to which such
%	additional directions are useful under non-constructed prior sampling is
%	investigated next.
	
}

{\color{black}
	
	\FloatBarrier
	\subsection{Accessible-subspace enlargement under random prior sampling}
	\label{sec:random-prior-robustness}
	
	We now test whether the accessible-subspace enlargement observed in
	Section~\ref{sec:numerics-subspace} persists when the truth, the initial
	positions, and the initial velocities are sampled independently, without
	using either the truth or the observation operator to construct the
	additional velocity directions.
	
	We consider the two-dimensional linear elliptic problem
	\begin{equation*}
		-\delta \Delta p + p = u
		\qquad \text{in } D=(0,1)^2,
		\qquad
		p=0
		\qquad \text{on } \partial D,
		%\label{eq:random-prior-elliptic}
	\end{equation*}
	with $\delta=2\cdot 10^{-3}$.
	The unknown forcing is discretized on a $64\times64$ grid, so that
	$d=4096$, while the state is observed at $16\times16$ uniformly distributed
	sensor locations, giving $K=256$ observations. We use an ensemble of
	$J=50$ particles, hence $J\ll d$.
	
	The truth and the initial ensembles are sampled independently from the same
	smooth Gaussian prior. More precisely, using the discrete sine basis indexed by the spatial
	frequencies $(m_1,m_2)$ in the two coordinate directions, the spectral
	amplitudes are weighted according to
	\begin{equation*}
		w_{m_1,m_2}
		=
		\left[
		1+(\pi\ell)^2(m_1^2+m_2^2)
		\right]^{-s/2},
		\qquad
		\ell=0.08,
		\qquad
		s=1.5,
		%\label{eq:random-prior-spectrum}
	\end{equation*}
	and the overall scale is normalized to give unit pointwise root-mean-square
	standard deviation. For each realization, $u^\dagger$, $X^0$, and a
	preliminary velocity ensemble are drawn independently from this prior.
	The velocities are then centered and rescaled so that their initial
	root-mean-square magnitude equals the spread of the position ensemble.
	In particular, neither the truth nor the observations are used to select
	their directions.
	
	Synthetic observations are generated with independent Gaussian noise whose
	standard deviation is $2\%$ of the root-mean-square noiseless signal.
	We perform $50$ independent realizations, including independent draws of the
	truth, position ensemble, velocity ensemble, and observational noise.
	
	For each realization, the spaces
	\[
	\mathcal S_x
	=
	\operatorname{span}
	\{x_j^0-\bar x^0\}_{j=1}^J,
	\qquad
	\mathcal S_{xv}
	=
	\operatorname{span}
	\{x_j^0-\bar x^0,v_j^0\}_{j=1}^J
	\]
	are computed only a posteriori. Their dimensions are $49$ and $98$,
	respectively, in all realizations. The median relative residual of the truth
	offset outside these spaces decreases from
	\[
	\operatorname{med}
	\frac{\|(I-P_{\mathcal S_x})(u^\dagger-\bar x^0)\|_h}
	{\|u^\dagger-\bar x^0\|_h}
	=
	0.8277
	\]
	to
	\[
	\operatorname{med}
	\frac{\|(I-P_{\mathcal S_{xv}})(u^\dagger-\bar x^0)\|_h}
	{\|u^\dagger-\bar x^0\|_h}
	=
	0.7385.
	\]
	Thus the independently sampled initial velocities enlarge the accessible
	space in directions that contain a non-negligible additional component of
	the truth, despite not using truth-dependent information in their
	construction.
	
	We compare standard EKI, a span-preserving inflated EKI, the second-order
	dynamics with zero initial velocities, and the second-order dynamics with
	the independently sampled nonzero initial velocities described above.
	As an anti-contraction baseline, inspired by standard ensemble inflation
	strategies \cite{AndersonAnderson1999} and related EKI stabilization approaches
	\cite{ArmbrusterHertyVisconti2022}, we augment the standard EKI particle
	update by the outward anomaly term
	\(\Delta t\,\rho(x_j^n-\bar x^n)\), with \(\rho=0.15\).
	Since this term belongs to the current anomaly span, the method preserves
	the initial affine subspace while counteracting ensemble contraction and
	does not enlarge the set of accessible directions.
	
	For both second-order variants we use
	\[
	\gamma=2,\qquad
	\beta=0.5,\qquad
	\alpha=0.4,\qquad
	k=0.1,\qquad
	\varepsilon=10^{-3},\qquad
	p=1.5.
	\]
	All parameters are kept fixed throughout the $50$ realizations, and no
	parameter is selected using the true solution. We use
	$
	\Delta t=0.1,
	\
	T=20.
	$
	
	Since the randomly generated accessible spaces need not contain a
	discrepancy-compatible solution, we first compute the exact restricted
	least-squares values
	$
	\Phi^\star(\mathcal S_x),
	%=
	%\min_{x\in\bar x^0+\mathcal S_x}\Phi(x),
	\
	\Phi^\star(\mathcal S_{xv}).
	%=
	%\min_{x\in\bar x^0+\mathcal S_{xv}}\Phi(x).
	$
	Over the $50$ realizations,
	\[
	\operatorname{med}
	\frac{\Phi^\star(\mathcal S_{x})}{\Phi_{\rm disc}}
	=
	403.37,
	\qquad
	\operatorname{med}
	\frac{\Phi^\star(\mathcal S_{xv})}{\Phi_{\rm disc}}
	=
	136.32.
	\]
	Moreover,
	$
	\Phi^\star(\mathcal S_{xv})>\Phi_{\rm disc}
	$
	in all $50$ realizations. Thus, unlike the controlled experiment of
	Section~\ref{sec:numerics-subspace}, even the enlarged space does not contain
	a discrepancy-compatible solution in the present random-prior experiment.
	Discrepancy reachability is therefore not an informative performance
	criterion here.
	
	Instead, we measure how much of the reduction in data misfit made
	geometrically available by the enlargement
	$\mathcal S_x\subset\mathcal S_{xv}$ is actually realized by the dynamics.
	We define
	\begin{equation*}
		\eta_\Phi
		=
		\frac{
			\Phi^\star(\mathcal S_{x})-\min_{0\le n\Delta t\le T}\Phi(\bar x^n)
		}{
			\Phi^\star(\mathcal S_{x})-\Phi^\star(\mathcal S_{xv})
		}.
		%\label{eq:available-misfit-gain}
	\end{equation*}
	Thus $\eta_\Phi=0$ corresponds to no improvement over the best fit available
	in the original position space, whereas $\eta_\Phi=1$ corresponds to
	realizing the full least-squares improvement available in
	$\mathcal S_{xv}$.
	
	Table~\ref{tab:random-prior-robustness} summarizes the results. The position-only methods remain confined to the original position space, whereas the second-order method with independently sampled nonzero velocities realizes approximately $87\%$ of the data-fit improvement made geometrically available by the enlarged space.
	
	\begin{table}[t]
		\centering
		\caption{Random-prior robustness experiment over $50$ independent
			realizations. Let $t_{\min\Phi}$ denote the discrete time at which
			$\Phi(\bar x)$ attains its minimum over the computed trajectory.
			The parameter error is evaluated a posteriori at $t_{\min\Phi}$.}
		\label{tab:random-prior-robustness}
		\small
		\begin{tabular}{lccc}
			\toprule
			Method
			& $\operatorname{med}\min_t\Phi/\Phi_{\rm disc}$
			& $\operatorname{med}\eta_\Phi$
			& $\operatorname{med} e_x(t_{\min\Phi})$ \\
			\midrule
			Standard EKI
			& 403.38 & $0$ & 0.9929 \\
			Inflated EKI
			& 403.37 & $0$ & 0.9942 \\
			Second-order, $V_0=0$
			& 403.37 & $0$ & 0.9943 \\
			Second-order, random $V_0$
			& 167.28 & 0.8673 & 0.8892 \\
			\bottomrule
		\end{tabular}
	\end{table}
	
	The improvement is consistent across realizations:
	$\eta_\Phi>0.75$ in $46$ out of $50$ trials. The median reconstruction
	error is also reduced from approximately $0.99$ for the position-only
	methods to $0.8892$ with random initial velocities.
	
	Together with Section~\ref{sec:numerics-subspace}, this shows that
	velocity-induced enlargement persists under independent prior sampling.
	However, $\dim(\mathcal S_{xv})=98\ll4096$ and the enlarged space remains	discrepancy-incompatible: random initial velocities enlarge the accessible geometry without removing the finite-ensemble restriction.
	
}

{\color{black}
	
	\subsection{Noise dependence and discrepancy-based regularization}
	\label{sec:noise-regularization}
	
	We next investigate the behavior of the interacting second-order EKI as a discrepancy-stopped inverse procedure as the observational noise level decreases.
	
	We use the same one-dimensional elliptic inverse problem introduced in
	Section~\ref{sec:numerics-subspace}, but replace the deliberately structured
	initialization used there by a smoothness-promoting Gaussian initialization.
	We take
	\[
	d=K=255,
	\qquad
	J=30,
	\qquad
	u^\dagger(s)=10\sin(8s).
	\]
	Let $S_\ell$ denote the normalized discrete sine modes and let
	\[
	B=
	\begin{pmatrix}
		S_1 & \cdots & S_{20}
	\end{pmatrix}
	\in\mathbb R^{d\times 20}.
	\]
	The initial ensemble is sampled in the fixed prior space
	$\operatorname{Range}(B)$. More precisely, writing $x_j^0=B a_j^0$, the
	coefficients are independent centered Gaussian variables with variances
	chosen so that the standard deviation of the physical amplitude of the
	$\ell$-th sine mode decays as $10/\ell$. Thus higher-frequency components are
	progressively penalized while the true eighth mode remains contained in the
	prior support.
	
	The initial velocities are independently sampled in the same prior space,
	centered over the ensemble, and rescaled so that
	\[
	\left(
	\frac1J\sum_{j=1}^J\|v_j^0\|_h^2
	\right)^{1/2}
	=0.5.
	\]
	The same initial positions and velocities are used throughout the entire
	noise study.
	
	We consider the six noise levels
	\[
	\sigma = 0.2 \cdot 2^{-q}, \quad q=0,\dots,5, \quad \ \text{with} \ \ \Gamma=\sigma^2 I.
	\]
	For each noise level we perform $20$ realizations. To obtain a paired
	comparison across noise levels, we first generate independent vectors
	$z_r\sim\mathcal N(0,I)$, $r=1,\ldots,20$, and use
	\[
	y_{\sigma,r}
	=
	Gu^\dagger+\sigma z_r.
	\]
	Hence, for a fixed realization $r$, only the magnitude of the noise is changed
	when $\sigma$ is varied.
%	The average value of $\Phi(u^\dagger)$ over the $20$
%	noise realizations is $127.80$, close to its expected value
%	$K/2=127.5$.
	
	All model parameters are kept fixed across noise levels and noise
	realizations:
	\[
	\gamma=2,
	\qquad
	\beta=0.5,
	\qquad
	\alpha=0.4,
	\qquad
	k=0.1,
	\qquad
	\varepsilon=10^{-3},
	\qquad
	p=1.5.
	\]
	In particular, no parameter is retuned as $\sigma$ decreases. The true
	parameter is used only a posteriori to evaluate the reconstruction error and
	does not enter either the dynamics or the stopping criterion.
	
	For this linear experiment, we use a semi-implicit variant of the scheme in
	Section~\ref{sec:numerical-method} in order to handle the increasing
	stiffness of the Kalman force as $\sigma$ decreases. Since both the initial
	positions and velocities belong to $\operatorname{Range}(B)$,
	Proposition~\ref{prop:generalized-subspace} implies that the full
	second-order dynamics remains in this space. The coefficient formulation
	used below is therefore only a change of coordinates, not an additional
	projection or model reduction.
	
	We thus write
	\[
	x_j^n=B a_j^n,
	\qquad
	v_j^n=B w_j^n,
	\qquad
	H=GB,
	\]
	where $a_j^n,w_j^n\in\mathbb R^{20}$ are the position and velocity
	coefficients, respectively. We denote by $\bar a^n$ and $C_a^n$
%	\[
%	\bar a^n
%	=
%	\frac1J\sum_{j=1}^J a_j^n,
%	\qquad
%	C_a^n
%	=
%	\frac1J\sum_{j=1}^J
%	(a_j^n-\bar a^n)\otimes(a_j^n-\bar a^n)
%	\]
	the empirical mean and covariance of the coefficient ensemble, respectively. In
	particular,
	\[
	C(X^n)=B C_a^n B^\top,
	\qquad
	C(X^n)G^\top=B C_a^n H^\top.
	\]
	
	We freeze $C_a^n$ and the interaction terms at time $t_n$, while treating
	the damping and the linear dependence of the Kalman term implicitly:
	\begin{align*}
		a_j^{n+1}
		&=
		a_j^n+\Delta t\,w_j^{n+1},
		%\label{eq:noise-imex-position}
		\\
		w_j^{n+1}
		&=
		w_j^n
		+\Delta t
		\left[
		-\gamma w_j^{n+1}
		+
		\frac{\beta}{\sigma^2}
		C_a^nH^\top
		\bigl(y-Ha_j^{n+1}\bigr)
		+
		F_j^n
		\right],
		%\label{eq:noise-imex-velocity}
	\end{align*}
	where $F_j^n$ denotes the coefficient representation of the attraction
	and repulsion terms evaluated at time $t_n$. Each time step therefore
	requires only the solution of a $20\times20$ linear system. We use the
	fixed step size
	$
	\Delta t=2.5\cdot10^{-2}
	$
	for every noise level.
	
	For every realization, we use the discrepancy stopping rule of
	Section~\ref{sec:numerical-method} with $\tau=1.05$. A maximum integration
	time $T_{\max}=1000$ is used only as a numerical search horizon. If the
	discrepancy threshold is not reached within this interval, the realization
	is recorded as unsuccessful.
	
	Figure~\ref{fig:noise-study-summary}(a) shows the relative parameter error at
	the discrepancy stopping time. The mean error over the successful
	realizations decreases from $0.623$ at $\sigma=0.2$ to $0.0965$ at
	$\sigma=0.00625$. A log--log fit over the six noise levels gives the empirical
	slope $0.49$. The reference line proportional to $\sigma^{1/2}$ in
	Figure~\ref{fig:noise-study-summary}(a) is included only as a visual guide;
	no theoretical convergence rate is claimed here.
	
	\begin{figure}[t]
		\centering
		\includegraphics[width=\textwidth]
		{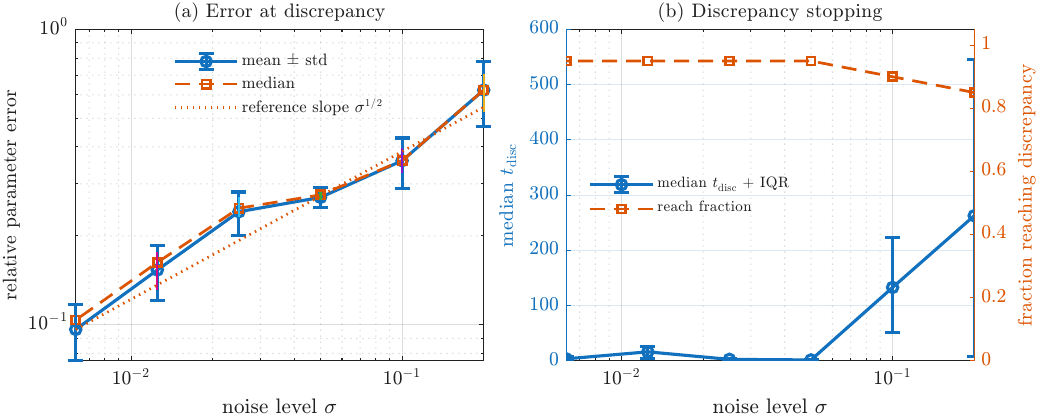}
		\caption{Noise dependence of the discrepancy-stopped reconstruction over
			$20$ paired noise realizations. Left: relative parameter error versus the
			noise standard deviation $\sigma$. Mean and standard deviation are shown
			together with the median; the dotted line has reference slope
			$\sigma^{1/2}$. Right: median discrepancy stopping time with interquartile
			range and fraction of realizations reaching the discrepancy threshold.}
		\label{fig:noise-study-summary}
	\end{figure}
	
	The fraction of realizations reaching the discrepancy threshold is,
	respectively,
	\[
	0.85,\quad
	0.90,\quad
	0.95,\quad
	0.95,\quad
	0.95,\quad
	0.95.
	\]
	Since $\tau=1.05$ gives a relatively tight discrepancy threshold, we report
	these unsuccessful runs explicitly rather than discarding them. Importantly,
	the decrease of the reconstruction error is not an artifact of the changing
	set of successful realizations. Restricting the comparison to the $17$
	paired realizations that reach the discrepancy threshold for all six noise
	levels gives essentially the same trend, with an empirical log--log slope
	approximately equal to $0.48$.
	
	The stopping times exhibit a skewed distribution, as shown in
	Figure~\ref{fig:noise-study-summary}(b). We therefore report their median and
	interquartile range rather than mean and standard deviation. In particular,
	the artificial integration time should not be interpreted as a monotone
	regularization parameter across different values of $\sigma$; the relevant
	comparison here is between the reconstructions selected by the same
	data-dependent discrepancy rule.
	
	Figure~\ref{fig:noise-study-reconstructions} shows representative
	discrepancy-stopped reconstructions for the same normalized noise realization,
	with visibly improved recovery as $\sigma$ decreases.
	
	\begin{figure}[t]
		\centering
		\includegraphics[width=0.92\textwidth]
		{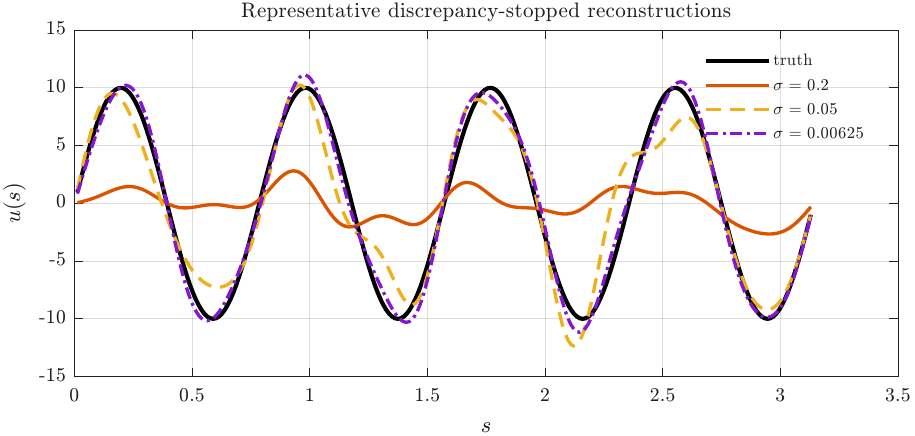}
		\caption{Representative discrepancy-stopped parameter reconstructions
			for the same normalized noise realization at
			$\sigma=0.2$, $0.05$, and $0.00625$. The stopping time is determined
			independently at each noise level by the discrepancy principle.}
		\label{fig:noise-study-reconstructions}
	\end{figure}
	
	Overall, with fixed initialization and algorithmic parameters, the observed
	improvement of the discrepancy-stopped reconstructions as the noise decreases
	is consistent with regularization behavior.
	
}

{\color{black}
	
	\FloatBarrier
	
	\subsection{Nonlinear Darcy inverse problem}
	\label{ssec:darcy}
	
	We finally consider a nonlinear and genuinely high-dimensional inverse
	problem based on Darcy flow.
%	Let
%	\[
%	D=(0,1)^2,
%	\]
%	and
	Consider
	\[
	-\nabla\cdot\bigl(\exp(u)\nabla p\bigr)=1
	\qquad \text{in } D=(0,1)^2,
	\qquad
	p=0
	\qquad \text{on } \partial D,
	\]
	where \(u:D\to\mathbb{R}\) denotes the log-permeability and
	\(p:D\to\mathbb{R}\) the pressure. The Darcy equation is discretized
	by a conservative five-point finite-difference scheme on the Cartesian
	grid, using harmonic averages of the permeability \(\exp(u)\) between
	neighboring grid points. The parameter \(u\) is discretized on a
	\(32\times32\) grid, giving \(d=1024\).
	The forward map
	$
	G:\mathbb{R}^{1024}\to\mathbb{R}^{64}
	$
	consists of pressure observations at \(K=64\) interior sensors
	arranged on an \(8\times8\) grid.
	
	The prior distribution for \(u\) is a smooth Gaussian random field represented
	in a tensor-product cosine basis. We use correlation scale \(0.2\), smoothness
	parameter \(2\), and rescale the field so that its pointwise root-mean-square
	standard deviation is \(0.3\). The truth \(u^\dagger\) is drawn independently
	from this prior. Observations are generated according to
	\[
	y=G(u^\dagger)+\eta,
	\qquad
	\eta\sim\mathcal{N}(0,\Gamma),
	\qquad
	\Gamma=\sigma^2 I_K,
	\]
	where \(\sigma\) is chosen as \(5\%\) of the root-mean-square magnitude of the
	noise-free observations.
	
	For the Darcy experiments we use the \(95\%\) Gaussian discrepancy threshold
	defined in Section~\ref{sec:numerical-method}. For \(K=64\),
	\[
	\Phi_{\rm disc}\simeq 41.84,
	\]
	while for the fixed noise realization used below,
	\[
	\Phi(u^\dagger)\simeq38.35<\Phi_{\rm disc}.
	\]
	
	All Darcy simulations use the time step
	\(\Delta t=2.5\cdot10^{-2}\); the second-order variants are discretized
	by the IMEX scheme of Section~4.1.
	Distances, spreads, and velocity root-mean-square values are computed using
	the discrete \(L^2(D)\) norm introduced in Section~\ref{sec:numerical-method};
	the corresponding inner product is also used to define orthogonal projections.
	
%	The three experiments below address, respectively, accessible-subspace
%	enlargement through nonzero initial velocities, the role of attraction and
%	repulsion in controlling ensemble geometry, and dependence on the ensemble
%	size \(J\).

	\subsubsection{Enlargement of the accessible subspace}
	\label{sec:darcy-enlargement}
	
	We first examine whether the generalized subspace mechanism of
	Proposition~\ref{prop:generalized-subspace} remains visible in the nonlinear
	Darcy problem. We use a small ensemble with \(J=10\). To construct the initial position
	and velocity spaces, we order the cosine modes of the Gaussian prior by
	decreasing spectral energy. The initial position ensemble is sampled using
	only the first eight modes. For the realized ensemble we then form
	\[
	\mathcal S_x
	=
	\operatorname{span}\{x_j^0-\bar x^0:j=1,\ldots,J\},
	\]
	which has dimension
	\[
	\dim(\mathcal S_x)=8.
	\]
	
	For the out-of-span velocity initialization, preliminary velocities are
	sampled using the next nine prior modes and then explicitly projected onto
	\(\mathcal S_x^\perp\). After centering, these provide nine independent
	additional directions, so that
	\[
	\dim(\mathcal S_{xv})=17,
	\qquad
	\dim(\mathcal S_{\rm new})=9,
	\qquad
	\mathcal S_{\rm new}
	:=
	\mathcal S_{xv}\cap\mathcal S_x^\perp.
	\]
	The spectral partition and all velocity directions are determined from the
	prior independently of the truth and the observations.
	
	To distinguish subspace enlargement from the mere introduction of kinetic
	energy, we compare three second-order initializations:
	\[
	V_0=0,
	\qquad
	V_0^{\rm in}\in\mathcal S_x,
	\qquad
	V_0^{\rm out}\in\mathcal S_x^\perp.
	\]
	The in-span velocities are independently sampled from the same first eight
	prior modes and projected onto the realized space \(\mathcal S_x\), whereas
	the out-of-span velocities are obtained from the next nine prior modes and
	projected onto \(\mathcal S_x^\perp\).
	The in-span and out-of-span velocity ensembles have the same root-mean-square
	magnitude,
	\[
	\|V_0^{\rm in}\|_{\rm RMS}
	=
	\|V_0^{\rm out}\|_{\rm RMS}
	=
	2.52\cdot10^{-2},
	\]
	corresponding to \(10\%\) of the initial ensemble spread,
	$
	S^0=2.52\cdot10^{-1}.
	$
	Hence the two nonzero-velocity initializations differ in the orientation,
	rather than the magnitude, of the initial kinetic perturbation.
	
	For all second-order runs we use
	\[
	\beta=1,
	\qquad
	\gamma=2,
	\qquad
	\alpha=2.5,
	\qquad
	k=k_J:=\frac{\kappa}{J},
	\qquad
	\kappa=5,
	\qquad
	\varepsilon=1,
	\qquad
	p=1.5.
	\]
	Since \(f(0)=1\), the full dynamics satisfies
	$
	k_JJf(0)-\alpha
	=
	\kappa f(0)-\alpha
	=
	2.5>0,
	$
	and therefore lies in the collapse-instability regime identified in
	Section~\ref{subsec:collapse-instability}. All methods are run with final search horizon \(T=10\).
	
	The enlargement is relevant for the chosen truth. The relative components
	of the truth offset lying outside the two accessible spaces are
	\[
	\frac{
		\|(I-P_{\mathcal S_x})(u^\dagger-\bar x^0)\|_h
	}{
		\|u^\dagger-\bar x^0\|_h
	}
	\simeq0.269,
	\qquad
	\frac{
		\|(I-P_{\mathcal S_{xv}})(u^\dagger-\bar x^0)\|_h
	}{
		\|u^\dagger-\bar x^0\|_h
	}
	\simeq0.177.
	\]
	Defining
	\[
	q^\dagger_{\rm new}
	:=
	P_{\mathcal S_{\rm new}}(u^\dagger-\bar x^0),
	\]
	the newly accessible component satisfies
	\[
	\frac{\|q^\dagger_{\rm new}\|_h}
	{\|u^\dagger-\bar x^0\|_h}
	\simeq0.203.
	\]
%	Thus a non-negligible part of the truth that is inaccessible from
%	\(\mathcal S_x\) lies in the additional velocity-generated directions.
	
	As a local data-space diagnostic, we linearize the forward map at
	\(\bar x^0\). For a linear subspace \(\mathcal S\), we define
	\[
	\Phi_{\rm lin}^\star(\mathcal S)
	:=
	\min_{z\in\mathcal S}
	\frac12
	\left\|
	y-G(\bar x^0)-DG(\bar x^0)z
	\right\|_\Gamma^2,
	\]
	namely the minimum data misfit attainable within
	\(\bar x^0+\mathcal S\) for the forward map linearized at \(\bar x^0\).
	For the two accessible spaces we obtain
	\[
	\frac{\Phi_{\rm lin}^\star(\mathcal S_x)}
	{\Phi_{\rm disc}}
	\simeq0.972,
	\qquad
	\frac{\Phi_{\rm lin}^\star(\mathcal S_{xv})}
	{\Phi_{\rm disc}}
	\simeq0.817,
	\]
	with corresponding data-space ranks \(8\) and \(17\). Hence enlargement is
	not required merely to make the discrepancy level locally reachable in this
	realization. Rather, the experiment tests whether the additional directions
	are actually explored and whether they contribute to the reconstruction.
	
	Figure~\ref{fig:darcy-enlargement} shows the resulting dynamics. For the
	out-of-span initialization, the relative components of the mean and ensemble
	anomalies outside \(\mathcal S_x\) reach approximately
	\(7.0\cdot10^{-2}\) and \(5.1\cdot10^{-1}\), respectively. At the same
	time, their residuals outside \(\mathcal S_{xv}\) remain at round-off level:
	the corresponding maxima are approximately \(1.3\cdot10^{-14}\) and
	\(1.1\cdot10^{-13}\). By contrast, the cases \(V_0=0\) and
	\(V_0^{\rm in}\in\mathcal S_x\) remain in \(\mathcal S_x\) up to round-off.
%	This directly verifies, in the nonlinear Darcy setting, that out-of-span
%	initial velocities allow the dynamics to leave the initial position space
%	while remaining confined to the generalized position--velocity space.
	
	\begin{figure}[t]
		\centering
		\includegraphics[width=\textwidth]
		{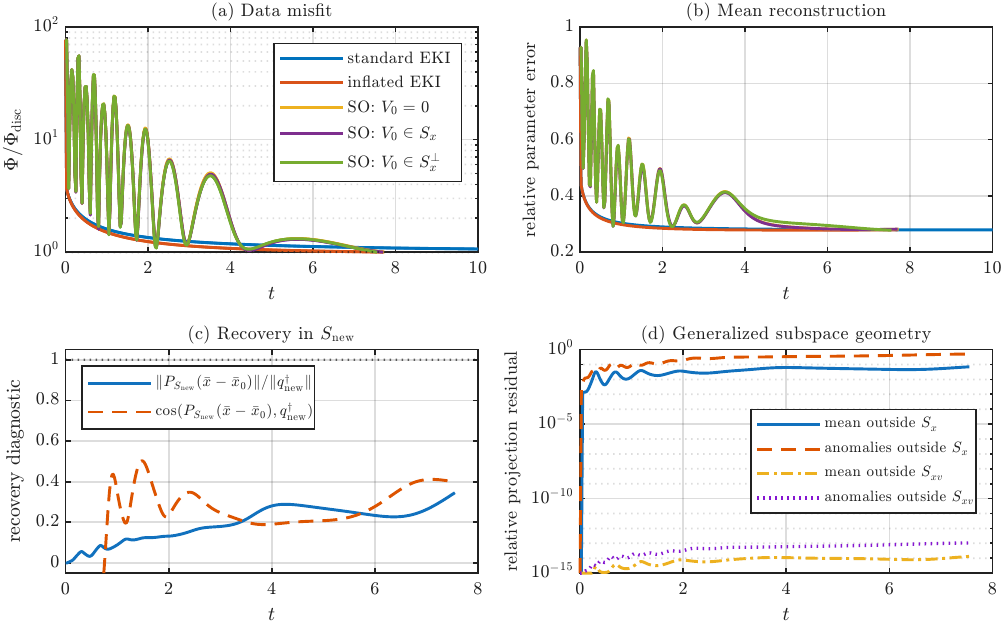}
		\caption{
			Darcy subspace-enlargement experiment with \(J=10\).
			(a) Normalized data misfit, with the discrepancy level indicated by
			the horizontal dotted line.
			(b) Relative reconstruction error of the ensemble mean.
			(c) Amplitude and alignment of the reconstructed component in the
			newly accessible space \(\mathcal S_{\rm new}\) for the out-of-span
			initialization.
			(d) Relative projection residuals of the mean and ensemble anomalies
			outside \(\mathcal S_x\) and \(\mathcal S_{xv}\), again for the
			out-of-span initialization.
		}
		\label{fig:darcy-enlargement}
	\end{figure}
	
	Table~\ref{tab:darcy-enlargement} summarizes the discrepancy-stopped
	results. Standard EKI does not reach the threshold, whereas the inflated EKI
	baseline and all three second-order variants do. The controls \(V_0=0\) and \(V_0^{\rm in}\in\mathcal S_x\) behave almost
	identically, while the out-of-span initialization yields only a modest
	improvement in global reconstruction error.
	
	\begin{table}[t]
		\centering
		\caption{
			Summary of the Darcy subspace-enlargement experiment.
			The reconstruction error and spread are evaluated at the discrepancy
			stopping time whenever the threshold is reached.
		}
		\label{tab:darcy-enlargement}
		\small
		\begin{tabular}{lccc}
			\toprule
			Method
			& Stopping time
			& Parameter error
			& Spread
			\\
			\midrule
			Standard EKI
			& not reached
			& -- & --
			\\
			Inflated EKI
			& \(7.700\)
			& \(0.2789\)
			& \(0.0229\)
			\\
			Second-order, \(V_0=0\)
			& \(7.725\)
			& \(0.2815\)
			& \(0.0783\)
			\\
			Second-order, \(V_0\in\mathcal S_x\)
			& \(7.700\)
			& \(0.2817\)
			& \(0.0783\)
			\\
			Second-order, \(V_0\in\mathcal S_x^\perp\)
			& \(7.550\)
			& \(0.2780\)
			& \(0.0915\)
			\\
			\bottomrule
		\end{tabular}
	\end{table}
	
	The newly accessible directions are not merely visited but also contribute
	to the reconstruction. For the out-of-span initialization, define
	\[
	E_{\rm new}(t)
	=
	\frac{
		\|P_{\mathcal S_{\rm new}}(\bar x(t)-\bar x^0)
		-q^\dagger_{\rm new}\|_h
	}{
		\|q^\dagger_{\rm new}\|_h
	}.
	\]
	At the discrepancy stopping time,
	\[
	E_{\rm new}=0.919,
	\qquad
	\frac{
		\|P_{\mathcal S_{\rm new}}(\bar x-\bar x^0)\|_h
	}{
		\|q^\dagger_{\rm new}\|_h
	}
	=0.346,
	\qquad
	\cos\!\left(
	P_{\mathcal S_{\rm new}}(\bar x-\bar x^0),
	q^\dagger_{\rm new}
	\right)
	=0.399.
	\]
	Thus the dynamics recovers a nonzero part of the missing truth component,
	but only with moderate alignment, consistently with the small improvement in
	the total reconstruction error.
	
	To investigate this limitation, we perform an a posteriori Jacobian-based
	diagnostic along the out-of-span trajectory. Let
	\[
	g_{\rm loc}
	=
	DG(\bar x)^\top\Gamma^{-1}
	\bigl(y-G(\bar x)\bigr)
	\]
	denote the local negative-gradient direction. Near the discrepancy stopping
	time,
	\begin{align*}
		\cos\!\left(
		P_{\mathcal S_{\rm new}}g_{\rm loc},
		q^\dagger_{\rm new}
		\right)
		&\simeq0.775,\\
		\cos\!\left(
		P_{\mathcal S_{\rm new}}C(X)g_{\rm loc},
		q^\dagger_{\rm new}
		\right)
		&\simeq0.239,\\
		\cos\!\left(
		P_{\mathcal S_{\rm new}}C(X)g_{\rm loc},
		P_{\mathcal S_{\rm new}}C^{xG}(X)
		\Gamma^{-1}(y-G(\bar x))
		\right)
		&\simeq0.999.
	\end{align*}
	Hence the local negative-gradient direction is substantially better aligned
	with the missing truth component than its covariance-preconditioned
	counterpart, while the latter is almost parallel to the empirical nonlinear
	Kalman direction. This indicates that, near the discrepancy level, the
	limited recovery in \(\mathcal S_{\rm new}\) is primarily associated with
	the geometry induced by covariance preconditioning rather than with the
	empirical cross-covariance approximation itself. This effect is
	time-dependent, and the covariance can rotate the local direction differently
	at earlier stages of the trajectory.
	
	The quantities involving \(q^\dagger_{\rm new}\) are truth-based and are used
	only as a posteriori diagnostics. The conclusion is therefore twofold:
	out-of-span initial velocities genuinely enlarge the accessible space and
	allow partial recovery along previously unavailable directions, but such
	enlargement does not by itself guarantee that the evolving covariance weights
	those directions favorably for truth reconstruction.
	
}

{\color{black}
	
	\subsubsection{Interaction ablation}
	\label{sec:darcy-interaction-ablation}
	
	We next isolate the role of attraction and repulsion in the second-order
	dynamics. We set
	$
	J=40,
	\
	V_0=0,
	$
	so that \(\mathcal S_{xv}=\mathcal S_x\). Consequently, all methods compared
	below share the same accessible affine space, and differences between them
	do not arise from velocity-induced subspace enlargement.
	
	To probe the interaction mechanism in a regime where covariance collapse may
	be relevant, we deliberately use an underdispersed initial ensemble. If
	\(\xi_j\) denotes a draw from the Gaussian prior introduced above, we set
	\[
	x_j^0
	=
	m_{\rm pr}
	+
	\delta_{\rm under}(\xi_j-m_{\rm pr}),
	\qquad
	\delta_{\rm under}=0.25,
	\]
	where \(m_{\rm pr}\) is the prior mean. Thus the position fluctuations are
	reduced by a factor \(0.25\), corresponding to a factor \(0.25^2\) in the
	initial empirical covariance. The resulting spread and entropy effective
	rank are
	\[
	S^0=7.83\cdot10^{-2},
	\qquad
	r_{\rm eff}(X^0)=10.15.
	\]
	This initialization is used as a controlled stress test for the interaction
	mechanism and should not be interpreted as a generic prior initialization.
	
	The initial anomaly space has dimension
	$
	\dim(\mathcal S_x)=39.
	$
	Using the local linearized diagnostic introduced in
	Section~\ref{sec:darcy-enlargement}, we obtain
	\[
	\frac{\Phi_{\rm lin}^\star(\mathcal S_x)}
	{\Phi_{\rm disc}}
	=0.436,
	\]
	with data-space rank \(39\). Thus the common accessible space is locally rich
	enough to contain a discrepancy-compatible correction, although this
	linearized diagnostic does not imply global reachability.
	
	We compare standard EKI, the span-preserving inflated EKI baseline introduced
	above, and four second-order variants: inertia only
	\((\alpha,k)=(0,0)\), inertia with attraction
	\((\alpha,k)=(2.5,0)\), inertia with repulsion
	\((\alpha,k)=(0,k_J)\), and the full interacting model
	\((\alpha,k)=(2.5,k_J)\), where
	$
	k_J=\kappa/J,
	\
	\kappa=5.
	$
	For all second-order variants,
	\[
	\beta=1,
	\qquad
	\gamma=2,
	\qquad
	\varepsilon=1,
	\qquad
	p=1.5.
	\]
	For the variants containing repulsion,
	$
	k_JJf(0)=\kappa f(0)=5.
	$
	Hence the full model satisfies
	$
	k_JJf(0)-\alpha=2.5>0
	$
	and lies in the collapse-instability regime of
	Section~\ref{subsec:collapse-instability}. Runs that do not reach the
	discrepancy threshold are continued up to \(T=30\).
	
	Figure~\ref{fig:darcy-interaction-ablation} shows three clearly distinct
	ensemble regimes. Attraction alone produces rapid contraction: by \(T=30\) the normalized
	spread has decreased to \(4.3\times10^{-13}\), corresponding to numerical
	ensemble collapse, and the discrepancy threshold is not reached. Repulsion
	alone produces the opposite behavior, with normalized spread \(8.96\) at
	\(T=30\); the ensemble remains non-collapsed but becomes strongly
	overdispersed, and its parameter error eventually increases.
	
	The full interacting dynamics balances these two effects. It reaches the discrepancy threshold while retaining a non-collapsed
	ensemble.
	Attraction therefore controls the excessive dispersion generated by repulsion
	without producing the collapse observed in the attraction-only variant.
	
	The full model attains the prescribed discrepancy substantially earlier in
	integration time than the three successful alternatives. This does not
	correspond to a smaller reconstruction error: at stopping, the full model
	has relative error \(0.305\), compared with approximately \(0.29\) for the
	three alternatives. The advantage observed here should therefore be
	interpreted in terms of ensemble geometry and discrepancy attainment,
	rather than as an improvement in truth reconstruction accuracy.
	
	\begin{figure}[t]
		\centering
		\includegraphics[width=\textwidth]
		{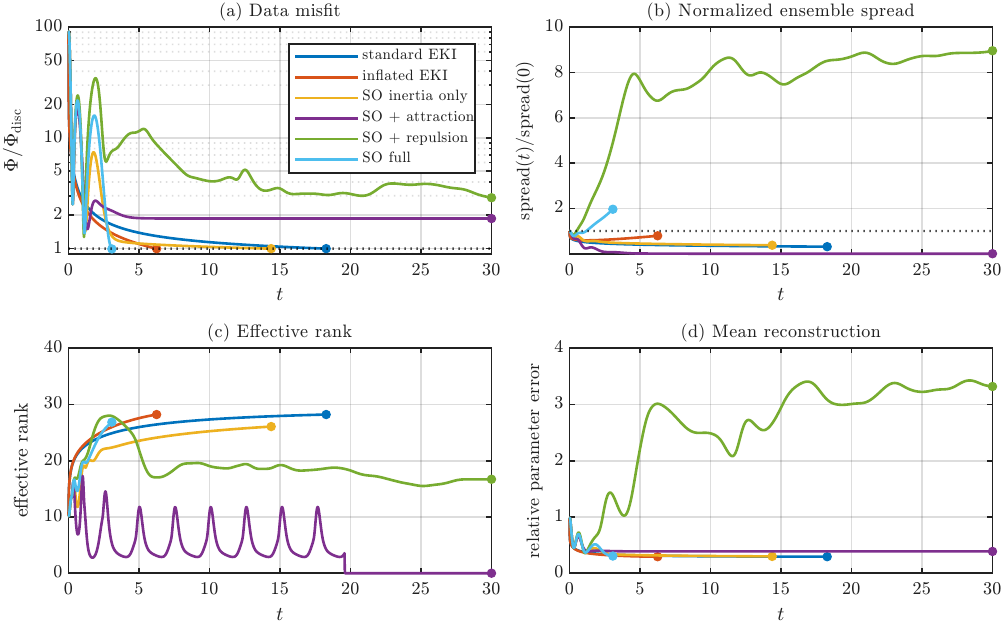}
		\caption{
			Darcy interaction-ablation experiment with \(J=40\) and
			\(V_0=0\).
			(a) Data misfit normalized by the discrepancy threshold.
			(b) Ensemble spread normalized by its initial value.
			(c) Entropy effective rank of the ensemble anomalies.
			(d) Relative reconstruction error of the ensemble mean.
			Circles indicate the first discrepancy stopping time when the threshold
			is reached and the final time \(T=30\) otherwise.
		}
		\label{fig:darcy-interaction-ablation}
	\end{figure}
	
	\begin{table}[t]
		\centering
		\caption{
			Summary of the Darcy interaction-ablation experiment.
			For successful runs, the parameter error, normalized spread, and
			effective rank are evaluated at the first discrepancy stopping time;
			for unsuccessful runs they are evaluated at \(T=30\).
			The minimum normalized misfit is computed over the full trajectory. An effective rank of zero denotes numerical ensemble collapse, according
			to the convention introduced in Section~\ref{sec:numerical-method}.
		}
		\label{tab:darcy-interaction-ablation}
		\small
		\begin{tabular}{lccccc}
			\toprule
			Method
			& Stopping time
			& \(\min_n \Phi(\bar x^n)/\Phi_{\rm disc}\)
			& Parameter error
			& \(S/S^0\)
			& \(r_{\rm eff}\)
			\\
			\midrule
			Standard EKI
			& \(18.275\)
			& \(0.9999\)
			& \(0.2938\)
			& \(0.311\)
			& \(28.21\)
			\\
			Inflated EKI
			& \(6.250\)
			& \(0.9986\)
			& \(0.2930\)
			& \(0.792\)
			& \(28.21\)
			\\
			Second-order, inertia only
			& \(14.375\)
			& \(1.0000\)
			& \(0.2987\)
			& \(0.374\)
			& \(26.07\)
			\\
			Second-order + attraction
			& not reached
			& \(1.4987\)
			& \(0.3871\)
			& \(4.31\cdot10^{-13}\)
			& \(0\)
			\\
			Second-order + repulsion
			& not reached
			& \(1.2726\)
			& \(3.3193\)
			& \(8.96\)
			& \(16.71\)
			\\
			Second-order, full
			& \(3.075\)
			& \(0.9937\)
			& \(0.3052\)
			& \(1.96\)
			& \(26.85\)
			\\
			\bottomrule
		\end{tabular}
	\end{table}
	
}

{\color{black}

	\FloatBarrier

	\subsubsection{Dependence on the ensemble size}
	\label{sec:darcy-ensemble-size}
	
	We finally investigate the dependence of the Darcy results on the ensemble
	size. We consider
	\[
	J\in\{10,20,40,80\},
	\]
	and compare standard EKI, the span-preserving inflated EKI baseline, the
	inertia-only second-order dynamics, and the full interacting model.
	For the
	latter we use the ensemble-size normalized repulsion strength introduced in
	Remark~\ref{rem:normalization:repulsion},
	$
	k=k_J:=\kappa/J,
	\
	\kappa=5.
	$
	By Remark~\ref{rem:repulsion:independent:J}, the corresponding
	collapse-instability condition is independent of \(J\).
	All remaining parameters are kept as in
	Section~\ref{sec:darcy-interaction-ablation}, and the initial velocities are
	set to zero.
	
	To isolate the effect of the ensemble size, we construct the initial
	ensembles from a common random pool. For every \(J\), the ensemble is
	recentered at the same prior mean and rescaled to have the same initial
	discrete \(L^2(D)\) spread,
	\[
	S^0=7.83\cdot10^{-2}.
	\]
	Thus increasing \(J\) adds ensemble directions without introducing a
	systematic change in either the initial mean or total spread. As in the
	preceding ablation study, this corresponds to
	\(\delta_{\rm under}=0.25\). Runs that do not reach the discrepancy threshold
	are continued up to \(T=30\).
	
	Table~\ref{tab:darcy-J-reachability} reports the corresponding
	accessible-space diagnostics. Since \(V_0=0\),
	\(\mathcal S_{xv}=\mathcal S_x\), whose dimension increases from \(9\) to
	\(79\). The rank of the linearized observation map restricted to
	\(\mathcal S_x\) increases accordingly and eventually saturates at the data
	dimension \(K=64\).
	
	\begin{table}[t]
		\centering
		\caption{
			Accessible-space diagnostics in the Darcy ensemble-size study.
			The last column gives the minimum local linearized misfit over
			\(\bar x^0+\mathcal S_x\), normalized by the discrepancy threshold.
		}
		\label{tab:darcy-J-reachability}
		\small
		\begin{tabular}{cccc}
			\toprule
			\(J\)
			& \(\dim(\mathcal S_x)\)
			& \(\dim(DG(\bar x^0)\mathcal S_x)\)
			& \(\Phi_{\rm lin}^\star(\mathcal S_x)/\Phi_{\rm disc}\)
			\\
			\midrule
			\(10\) & \(9\)  & \(9\)  & \(1.204\) \\
			\(20\) & \(19\) & \(19\) & \(0.842\) \\
			\(40\) & \(39\) & \(39\) & \(0.435\) \\
			\(80\) & \(79\) & \(64\) & \(<10^{-20}\) \\
			\bottomrule
		\end{tabular}
	\end{table}
	
	For \(J=10\), the local linearized minimum remains above the discrepancy
	threshold and none of the four methods reaches discrepancy within \(T=30\).
	This is consistent with the severe finite-ensemble restriction, although the
	linearized diagnostic is local and does not provide a global
	non-reachability result.
	
	At \(J=20\), only inflated EKI and the full interacting dynamics reach the	discrepancy, whereas for \(J=40\) and \(80\) all four methods do. Once the ensemble is sufficiently rich, the full model reaches the prescribed discrepancy earlier in integration time than the inflated baseline.
	
	Since integration time alone does not measure computational work, we also
	compare the number of forward-model evaluations in
	Figure~\ref{fig:darcy-ensemble-size}(b). For the full and inflated methods,
	respectively, the forward-solve counts are
	\[
	6070\ \text{vs.}\ 7476
	\quad (J=20),\qquad
	4962\ \text{vs.}\ 10414
	\quad (J=40),\qquad
	8830\ \text{vs.}\ 20007
	\quad (J=80).
	\]
	Thus the full dynamics requires approximately \(19\%\), \(52\%\), and
	\(56\%\) fewer forward solves than the inflated baseline at these three
	ensemble sizes.
	
	For completeness, the corresponding wall-clock measurements are also
	reported. At \(J=40\), the runtimes for standard EKI, inflated EKI, inertia
	only, and the full model are \(9.20\), \(3.12\), \(7.20\), and \(1.52\)
	seconds, respectively; at \(J=80\), they are \(12.59\), \(4.38\), \(13.37\),
	and \(2.06\) seconds. At \(J=20\), the two successful methods, inflated EKI
	and the full model, require \(3.01\) and \(2.69\) seconds. As discussed in
	Section~\ref{sec:numerical-method}, these MATLAB Online timings are
	implementation dependent and include the additional \(O(J^2)\) cost of the
	direct pairwise interaction; they are therefore used only as a complement
	to the forward-solve counts.
	
	\begin{figure}[t]
		\centering
		\includegraphics[width=\textwidth]{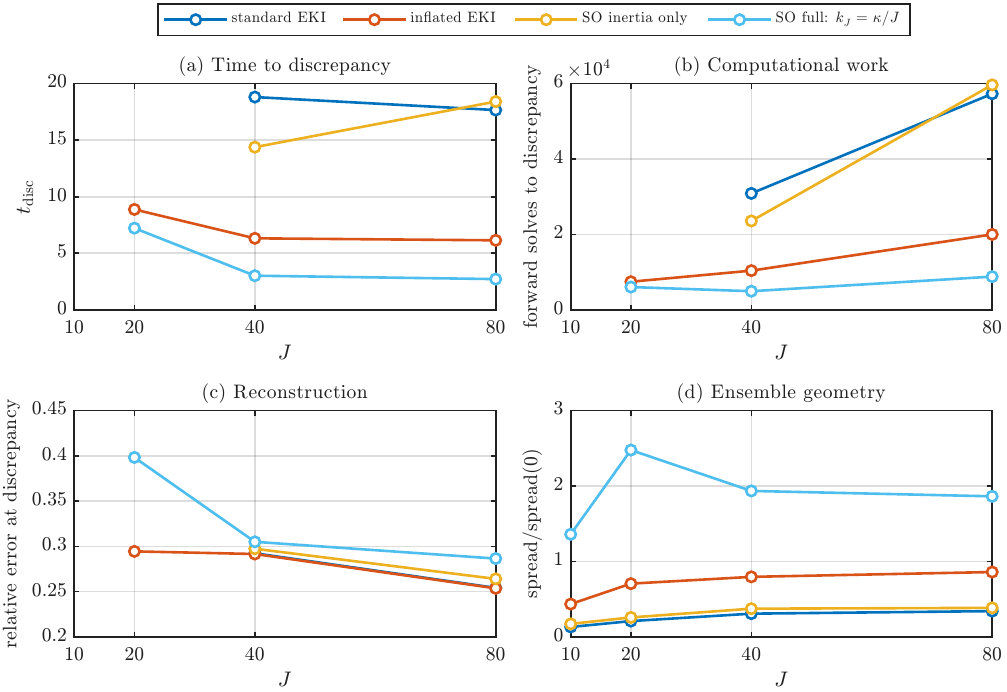}
		\caption{
			Darcy ensemble-size study for
			\(J\in\{10,20,40,80\}\).
			(a) First discrepancy stopping time.
			(b) Number of forward solves required to reach the discrepancy.
			(c) Relative parameter error at the discrepancy stopping time.
			(d) Ensemble spread \(S/S^0\), evaluated at the discrepancy stopping
			time when the threshold is reached and at \(T=30\) otherwise.
			Values in panels (a)--(c) are omitted whenever the discrepancy
			threshold is not reached.
		}
		\label{fig:darcy-ensemble-size}
	\end{figure}
	
	The reduction in computational work does not correspond to a systematic
	improvement in truth reconstruction. At each successful ensemble size, the full model has a larger
	discrepancy-stopped reconstruction error than inflated EKI; see
	Figure~\ref{fig:darcy-ensemble-size}(c).
	The advantage of the full dynamics should therefore be associated with its
	ensemble evolution and efficient discrepancy attainment, rather than with
	uniformly improved parameter reconstruction.
	
	Finally, Figure~\ref{fig:darcy-ensemble-size}(d) shows that standard EKI and
	inertia only contract substantially, whereas inflated EKI counteracts this
	contraction. The normalized full model instead maintains a comparably dispersed,
	non-collapsed ensemble across the successful values of \(J\).
	
}

\FloatBarrier
\section{Conclusions}
\label{sec:conclusions}

We introduced a second-order Ensemble Kalman particle system that augments
continuous-time EKI with inertia, damping, attraction towards the ensemble
mean, and short-range repulsion. The resulting dynamics provides a
heavy-ball-type formulation in which the internal ensemble geometry is
modified explicitly by competing attractive and repulsive mechanisms.

{\color{black}
For linear inverse problems, we established a generalized subspace property
showing that nonzero initial velocities may enlarge the accessible affine
space, while a finite-ensemble restriction remains. We also identified a
parameter regime in which fully collapsed configurations are linearly
unstable with respect to zero-mean fluctuation perturbations. Assuming
convergence to an asymptotic equilibrium, we characterized the corresponding
optimality condition on the subspace retained by the limiting covariance and,
for the associated frozen-covariance mean dynamics, proved exponential decay
along the retained directions. Numerically, we observed that the second-order
terms can improve ensemble exploration and discrepancy attainment in
high-dimensional inverse problems, while not yielding uniform gains in
reconstruction accuracy.

The proposed method is therefore not intended as a uniformly superior
replacement for first-order EKI. In practice, ensemble diagnostics such as
the data misfit, spread, and effective rank may help identify premature
contraction, although we do not propose an automatic switching or parameter
selection rule. Likewise, the usefulness of velocity-induced subspace
enlargement remains problem dependent. The present formulation is designed
for inversion and optimization rather than posterior sampling; extending the
approach to Bayesian uncertainty quantification would require a separate
design and analysis.
}

\section*{Acknowledgments}

This work was carried out within the activities of the PRIN PNRR Project 2022 No.~P2022JC95T,
``Data-driven discovery and control of multiscale interacting artificial agent systems'',
funded by MUR (Ministry of University and Research) and Next Generation EU -- European Commission.

G.V. acknowledges the support of Sapienza University under Ateneo Project 2024
``Advanced Computational Methods for Real-World Applications: Data-Driven Models, Hyperbolic Equations and Optimal Control''.

G.V. is a member of the INdAM Research National Group of Scientific Computing (INdAM--GNCS) and of the SIMAI's Activity Group on ``Multiscale Modelling of Interacting Agents''.

{\color{black} The authors thank the anonymous referees for their careful reading and	constructive comments, which helped improve the manuscript.}

\bibliographystyle{abbrv}
\bibliography{references}

\end{document}